\documentclass[12pt]{article}

\usepackage{amsmath}
\usepackage{amssymb}
\usepackage{amsthm}

\usepackage{graphicx}
\usepackage[arrow, matrix, curve]{xy}

\newtheorem{theorem}{Theorem}

\newtheorem{lemma}[theorem]{Lemma}
\newtheorem{corollary}[theorem]{Corollary}
\newtheorem{proposition}[theorem]{Proposition}

\newtheorem{example}{Example}
\theoremstyle{remark}\newtheorem{remark}[theorem]{Remark}
\theoremstyle{definition}\newtheorem{definition}[theorem]{Definition}

\newcommand{\N}{\mathbb N}
\newcommand{\R}{\mathbb R}

\newcommand{\im}{\operatorname{im}}

\newcommand{\aplim}{\operatorname{aplim}}

\def\ka{\kappa}

\def\ga{\gamma}
\def\Ga{\Gamma}

\def\eps{\epsilon}

\def\si{\sigma}
\def\Si{\Sigma}
\def\om{\omega}
\def\Om{\Omega}
\def\<{\langle}
\def\>{\rangle}

\def\3{\ss}

\def\a{\operatorname{area}}
\def\length{\operatorname{length}}

\def\D{\partial}

\def\2pithird{\frac{2\pi}{3}}

\newcommand{\ra}{\rightarrow}

\def\wlim{\operatorname{\omega-lim}}

\newcommand\blfootnote[1]{%
  \begingroup
  \renewcommand\thefootnote{}\footnote{#1}%
  \addtocounter{footnote}{-1}%
  \endgroup
}

\newcommand{\Addresses}{\bigskip\footnotesize 
\par\nopagebreak\textsc{Mathematisches Institut der Universit\"at M\"unchen, Theresienstr. 39, D-80333 M\"unchen, Germany}
\par\nopagebreak
\textit{Email}: \texttt{stadler@math.lmu.de}}

\begin{document}

\title{The structure of minimal surfaces in CAT(0) spaces}

\author{Stephan Stadler }

\date{\today}

\maketitle


\begin{abstract}
We prove that a  minimal disc in a  CAT(0) space is a local embedding away from a finite set
 of "branch points". On the way we establish several basic properties of minimal surfaces:
monotonicity of area densities, density bounds, limit theorems and the existence of
tangent maps. 

As an application, we prove F\'{a}ry-Milnor's theorem in the CAT(0) setting.
\end{abstract}

\section{Introduction}

\subsection{Motivation and main results}

Minimal surfaces are an indispensable tool in Riemannian geometry. Part of their success relies on the
well understood structure of minimal discs.
For example, by the classical Douglas-Rado Theorem, any smooth Jordan curve $\Ga$ in $\R^n$ bounds a least-area disc.
Moreover, this disc is a smooth immersion away from a finite set of branch points. Recently, Alexander Lytchak
and Stefan Wenger proved that a rectifiable Jordan curve in a proper metric space bounds a least-area disc as long as it bounds at least one
disc of finite energy \cite{LWplateau}. As in the case of Douglas-Rado, the minimal disc is obtained by minimizing energy among all admissible boundary parametrizations. 
The existence and regularity of energy minimizers or {\em harmonic maps} in metric spaces was studied earlier, usually under some kind of nonpostive curvature assumption \cite{GS}, \cite{KS}, \cite{J}.
For instance, Nicholas Korevaar and Richard Schoen solved the Dirichlet problem in CAT(0) spaces and showed that the resulting harmonic maps are locally Lipschitz in the interior \cite{KS}.

The intrinsic geometry of minimal discs was studied in \cite{M}, \cite{LWcurv} and \cite{PS}. There, it is shown (with varying generality) that minimal discs in CAT(0) spaces are intrinsically
nonpositively curved. However,
apart from regularity nothing else is known about the mapping behavior of minimal discs.
The first aim of this paper is to establish topological properties. We obtain the following structural result for minimal discs, similar
to the classical statement that minimal surfaces are branched immersions.

\begin{theorem}\label{thm:one}
Let $X$ be a CAT(0) space and $\Ga\subset X$ a rectifiable Jordan curve. Let $u:D\to X$ be a minimal disc filling $\Ga$.
Then there exists a finite set $B\subset D$ such that $u$ is a local embedding on $D\setminus B$.
\end{theorem}

Unlike in the smooth case, the corresponding result for harmonic discs fails, even if the target is of dimension two and has only isolated singularities, see work of Ernst Kuwert \cite{Ku}. 
 
We then aim at topological applications and prove the F\'{a}ry-Milnor Theorem for CAT(0) spaces, generalizing the original theorem proved independently
 by Istvan F\'{a}ry \cite{Fa} and John Milnor \cite{Mil}. 

\begin{theorem}(F\'{a}ry-Milnor)\label{thm:two}
Let $\Ga$ be a Jordan curve in a CAT(0) space $X$. If the total curvature of $\Gamma$ is less than $4\pi$,
then $\Ga$ bounds an embedded disc.
\end{theorem}

See Theorem \ref{thm:fm} for a more general result and Subsection \ref{subsec:tc} for the definition of total curvature.

Our proofs of both theorems rely heavily on the monotonicity of area densities:

\begin{theorem}[Monotonicity]
Let $X$ be a CAT(0) space. Suppose that $u:\bar{D}\to X$ is a minimal disc and $p$
is a point in $u(\bar{D})\setminus u(\partial\bar{D})$. Then the area density
$$
\Theta(u,p,r):=\frac{\a(u(D)\cap B_r(p))}{\pi r^2}
$$
is a nondecreasing function of $r$ as long as
$r<|p,u(\partial D)|$.  
 
\end{theorem}

\subsection{Overview and further results}

{\em On the structure of minimal surfaces.}

In order to obtain control on the mapping behavior of minimal discs we make intensive use of the intrinsic point of view developed by
Alexander Lytchak and Stefan Wenger in \cite{LWint} and \cite{LWcurv}. (See also \cite{PS}.) They
 showed that if $X$ is CAT(0), then any 
minimal disc $u:D\to X$ factors through an intrinsic space $Z_u$, which is itself a CAT(0) disc. 
Furthermore, it turns out that the factors $\pi:D\to Z_u$
and $\bar u:Z_u\to X$ are particularly nice. Namely, $\pi$ is a homeomorphism and $\bar u$ preserves the length of every rectifiable curve, cf. Theorem \ref{thm:int}.
Therefore, in order to prove Theorem \ref{thm:one}, we only need to investigate the induced map $\bar u$. For this purpose we introduce the notion of 
{\em intrinsic minimal surfaces} which by definition is a synthetic version of the induced map $\bar u$. We then prove several
basic properties for intrinsic minimal surfaces, all well-known in the smooth case.
The most important of these is the monotonicity of area densities and a corresponding lower density bound, 
(see Proposition \ref{prop:mon} and Lemma \ref{lem:estinv}). As in the classical case, monotonicity is accompanied by a rigidity statement (Theorem \ref{thm:rigid}).

However, the proof of rigidity is more involved and we were unable to directly derive it from the monotonicity of area densities.
Instead, we first investigate intrinsic minimal surfaces on an infinitesimal scale. We show that intrinsic minimal surfaces have tangent maps at all points.
Tangent maps for harmonic maps into CAT(0) spaces were also investigated by Misha Gromov and Richard Schoen \cite{GS} and later by Georgios Daskalopoulos
and Chikako Mese \cite{DM}. However, our results are independent, as we aim for "intrinsic tangent maps".
In our setting, we show that each such tangent map is itself an intrinsic minimal map which in addition is conical and even locally isometric away from a single point. (See Lemma \ref{lem:blowup} and the prior definition.)
Building on this we prove the rigidity supplement to monotonicity and our main structural result:

\begin{theorem}\label{thm:structure}
  Let $X$ be a CAT(0) space and $Z$ a CAT(0) disc. Suppose that $f:Z\to X$ is an intrinsic minimizer. 
  If $z_0$ is a point in the interior of $Z$ with $\mathcal{H}^1(\Sigma_{z_0}Z)<4\pi$, then $f$ restricts to a bilipschitz embedding on a 
  neighborhood of $z_0$. In particular, if $Z$ is a CAT(0) disc, then $f$ is locally a bilipschitz embedding in the interior of $Z$ away form 
  finitely many points.
 \end{theorem}
 
Together with Theorem \ref{thm:int} this then yields Theorem \ref{thm:one}.

{\em On F\'{a}ry-Milnor's theorem.}

The original theorem of F\'{a}ry-Milnor from 1949 says that a knot in $\R^3$ has to be the unknot if it is of finite total curvature less or equal
than $4\pi$. The first generalization of this theorem to variable curvature came about 50 years later and is due to Stefanie Alexander and Richard Bishop (\cite{AB}). 
We also refer the reader to their work for the history of the problem. Their result extended the F\'{a}ry-Milnor Theorem to simply connected 3-dimensional manifolds of nonpositive sectional curvature.
More precisely, 
it is shown in \cite{AB} that a Jordan curve of total curvature less or equal to $4\pi$ in a 3-dimensional Hadamard manifold bounds an embedded disc.
In the same paper it was noticed that the analog statement cannot be true for CAT(0) spaces. There is an example of a Jordan curve 
of total curvature $4\pi$ in a 2-dimensional CAT(0) space which does not bound an embedded disc, see Example \ref{ex}.
However, Theorem \ref{thm:two} shows that F\'{a}ry-Milnor's theorem does hold for CAT(0) spaces, at least if the strict inequality for the total curvature is fulfilled.
We actually prove the following more general result.

\begin{theorem}(Rigidity case of F\'{a}ry-Milnor)\label{thm:fm}
Let $\Ga$ be a Jordan curve in a CAT(0) space $X$. If the total curvature of $\Ga$ is less or equal to $4\pi$,
then either $\Ga$ bounds an embedded disc, or else the total curvature is equal to $4\pi$ and $\Ga$
bounds an intrinsically flat geodesic cone. More precisely, there is a map from a convex subset of a Euclidean cone of cone angle
$4\pi$ which is a local isometric embedding away from the cone point and which fills $\Ga$.
\end{theorem}  

Our proof relies on minimal surface theory and  follows the strategy of Tobias Ekholm, Brian White, and Daniel Wienholtz in \cite{EWW}, 
where the authors show that a minimal surface $\Sigma$ in $\R^n$ 
of any topological type is embedded if the total curvature $\kappa$ of the boundary is less than or equal to $4\pi$. Their approach was also used
in \cite{CG} to prove the F\'{a}ry-Milnor theorem in n-dimensional Hadamard manifolds.

We quickly recall their argument. If $\Sigma$ is such a surface in $\R^n$, then for any point $p$ not in the boundary of $\Sigma$ one augments
$\Sigma$ by an exterior cone $E_p$ over $\D\Sigma$. More precisely, $E_p=\bigcup_{q\in\D\Sigma}\{p+t(q-p):\ t\geq 1\}$.
The monotonicity of area densities continues to hold for $\Sigma\cup E_p$ and now it even holds for all times. 
Since the area growth of $E_p$ is equal to $\kappa$, this relates the
number of inverse images of $p$ to the total curvature of $\D \Sigma$. The completion of the proof is then based on  
a lower density bound.

In our case there is no exterior cone. Additionally, for an ordinary minimal disc in a CAT(0) space the required  lower
density bound is unclear. However, for intrinsic minimal discs the lower density bound is obvious and the above argument still shows the following (Corollary \ref{lem:eucareagrowth}). 
    
\begin{theorem}
 Let $\hat X$ be a CAT(0) space and $\hat f:\hat Z\to\hat X$ a proper intrinsic minimal plane. Suppose that the area growth of  $\hat f$
 is less than twice the area growth of the Euclidean plane. Then $\hat f$ is an embedding.
\end{theorem}

In order to prove the F\'{a}ry-Milnor Theorem we then show an extension result for intrinsic minimal discs. Roughly, it says that 
for each intrinsic minimal disc $f:Z\to X$ with finite total curvature $\kappa$ of the boundary we can embed $X$ isometrically into a 
CAT(0) space $\hat X$ such that $f$ extends to a proper intrinsic minimal plane $\hat f$ in $\hat X$ (Proposition \ref{prop:extension}).
The space $\hat X$ is obtained from $X$ by gluing a flat funnel along the boundary of $f$. 

\medskip

{\bf Acknowledgements.} I would like to thank Alexander Lytchak for many helpful discussions and his valuable comments on 
an earlier version of this paper.

\medskip

\tableofcontents

\section{Preliminaries from metric geometry}

 We refer the reader to 
\cite{BBI}, respectively \cite{Bmet}, for definitions and basics on metric geometry and to \cite{Bnpc},\cite{BH} and \cite{KL} for metric spaces with upper curvature bounds. 
 However, we include this short section to agree on some terminology and notations.

\subsection{Generalities} 
 
Let $D$ be the open unit disc in the plane and denote by $S^1$ the unit circle.

For a metric space $X$ we will denote the distance between two points $x,y\in X$ by $|x,y|$,
i.e. $|\cdot,\cdot|$ is the metric on $X$. If $\lambda>0$ we define the rescaled metric space $\lambda\cdot X$
be declaring the distance between points to be $\lambda$ times their old distance.

For a subset $A\subset X$ we denote by $\bar A$ its closure. 
If $x\in X$ is a point and $r>0$ is a radius, we denote by $B_r(x)$ the open ball of radius $r$ around $x$ in $X$.
More generally, for a subset $P\subset X$ we denote by $N_r(P)$ the tubular neighborhood of radius $r$.

A {\em Jordan curve} in $X$ is a subset $\Ga\subset X$ which is homeomorphic to $S^1$. If $\Ga$ is a Jordan curve in $\R^2$, then
its {\em Jordan domain} is the bounded connected component of $\R^2\setminus \Ga$.

A geodesic in $X$ is a curve of constant speed whose length is equal to the distance of its endpoints. 
A space is called {\em geodesic} , if there is a geodesic between any two of its points and it is called {\em uniquely geodesic}
if there is only one such geodesic.
If $X$ is uniquely geodesic, then $xy$ will denote the
(image) of the unique geodesic between $x$ and $y$. For a point $p\in X$ we call a geodesic {\em $p$-radial}, if it starts in $p$. A map $f:X\to Y$ to another metric space $Y$
is called {\em radial isometry (with respect to $p$)}, if it preserves distances to $p$.
As usual, a  map $f:X\to Y$ between metric spaces $X$ and $Y$ is called {\em Lipschitz continuous} or simply {\em Lipschitz}, if there exists a (Lipschitz-)constant
$L>0$ such that $|f(x),f(x')|\leq L\cdot|x,x'|$ holds for all $x,x'\in X$. Further, $f$ will be called {\em bilipschitz}, if it is bijective and has a Lipschitz continuous inverse.
Distance nonincreasing maps or $1$-Lipschitz maps between metric spaces will simply be called {\em short}. For $n\in\N$ we will denote by $\mathcal{H}^n$ the $n$-dimensional
Hausdorff measure on $X$.

A {\em surface} is a $2$-dimensional topological manifold, possibly with boundary. If $\Si$ is a surface and $f:\Si\to X$
is a map into a space $X$, then we denote by $\partial f:\partial \Si\to X $ the restriction $f|_{\D \Si}$.

If $(X_k,x_k)$ denotes a pointed sequence of metric spaces we can always take an {\em ultralimit} $(X_\om,x_\om)$ with respect to some
non-principal ultrafilter $\om$ on the natural numbers. If $(Y_k,y_k)$ denotes another such sequence  and $f_k:X_k\to Y_k$
are Lipschitz maps with a uniform Lipschitz constant $L>0$, the we obtain an $L$-Lipschitz ultralimit $f_\om:(X_\om,x_\om)\to(Y_\om,y_\om)$. 
For a precise definition and basics on ultralimits of metric spaces we refer the reader to
\cite{AKP}. 

\subsection{Intrinsic metric spaces}

A metric space $X$ is called {\em intrinsic metric space} or {\em length space}, if the distance between any two of its points is equal to the greatest lower bound
of the lengths of continuous curves joining those points.
Let $M$ be a topological space and $X$ a metric space. If $f:M\to X$ is a continuous map, then we can define on $M$ an {\em intrinsic (pseudo-)metric}
associated to $f$. Namely the intrinsic distance between two points in
$M$ is equal to the greatest lower bound for lengths of $f$-images of curves joining these points. If any pair of points in $M$
is joined by a curve whose image under $f$ is rectifiable, then one can identify points of zero $f$-distance in $M$ to obtain an associated {\em intrinsic
metric space} $M_f$. For instance, this is ensured if $M$ is a length space and $f$ is Lipschitz continuous. 
We will say that a map $f$ has some property {\em intrinsically}, if the associated space $M_f$ has this property.
If $M$ is equal to $X$ and $f$ is the identity, then we obtain the {\em intrinsic space} associated to $X$ and we denote it by $X^i$.

\subsection{CAT($\kappa$)}

If $X$ is a CAT($\kappa$) space and $p,x,y\in X$
are points such that $p\neq x,y$, then there is a well defined angle $\angle_p(x,y)$ between $x$ and $y$ at $p$.
Each point $p$ in a CAT($\kappa$) space $X$ has an associated {\em space of directions} or {\em link} $\Sigma_p X$ which is
the metric completion, with respect to angles, of germs of geodesics starting at $p$. Recall that
$\Sigma_p X$ is a CAT(1) space with respect to the intrinsic metric induced by $\angle$.
The {\em tangent space at $p$} is defined as the Euclidean cone over $\Sigma_p X$ and denoted by $T_p X$.
In particular, $T_p X$ is again a CAT(0) space. Note that if $X$ is CAT(0), then there is a natural short {\em logarithm map}
$\log_p:X\to T_p X$ which is a radial isometry and preserves initial directions of $p$-radial geodesics. 

The following two theorems by Reshetnyak will be used repeatedly. 
The gluing theorem is useful in order to check if a certain space is CAT(0).
For a detailed discussion and a proof we refer to \cite{Bmet} and \cite{AKP}.
The majorization theorem provides controlled Lipschitz fillings of circles.

\begin{theorem}(Reshetnyak's gluing theorem)\label{thm:glue}
Let $X$ and $X'$ be two CAT(0) spaces with closed convex subsets $C\subset X$ and $C'\subset X'$.
If $\iota:C\to C'$ is an isomety, then the space $X\cup_\iota X'$, which results from gluing $X$ and $X'$ via $\iota$, is CAT(0) with respect to the 
induced length structure. 
\end{theorem}

\begin{theorem}(Reshetnyak's majorization theorem)\label{thm:maj}
 Let $X$ be a CAT(0) space and $c:[0,L]\to X$ a closed curve of unit speed.
 Then there is a convex region $C\subset\R^2$, possibly degenerated,
 bounded by a unit speed curve $\tilde c:[0,L]\to\R^2$, and a short map $\mu:C\to X$
 with $\mu\circ\tilde c=c$.
\end{theorem}

As a consequence, the Euclidean isoperimetric inequality for curves holds in CAT(0) spaces.

\begin{theorem}(Isoperimetric inequality)\label{thm:isoin}
 Let $X$ be a CAT(0) space and $c:S^1\to X$ a Lipschitz circle. Then there exsits a Lipschitz extension
 $\hat c:\bar D\to X$ of $c$ with 
 
 \[\a(\hat c)\leq \frac{\length(c)^2}{4\pi}\]
 
 where the area of $\hat c$ is the Hausdorff $2$-measure counted with multiplicities, cf. Definition \ref{def:area}.
\end{theorem}

\subsection{Total curvature}\label{subsec:tc}

Let $X$ be a CAT(0) space. A curve $\si:[a,b]\ra X$ is called a {\em k-gon}, if there is a subdivision $a=t_0<t_1<\ldots<t_k=b$ such that
the restrictions $\si|_{[t_i,t_{i+1}]}$ are geodesics. Note that every ordered $k$-tuple $(x_1,\ldots,x_k)$ of points $x_i\in X$ 
determines an $k$-gon. The points $x_i:=\si(t_i)$ 
are called {\em vertices} of $\si$. If the number of vertices is not important, we will simply call $\si$ a {\em polygon}. 
The {\em total curvature} $\ka(\si)$ of a k-gon $\si$ with 
vertices $(x_i)$ is defined by
$$
\ka(\si):=\sum_{i=2}^{k-1}\left(\pi-\angle_{x_i}(x_{i-i},x_{i+1})\right).
$$

\noindent A k-gon $\si$ with 
vertices $(x_i)$ is {\em inscribed} in a curve $\ga$, if there is a parametrization $\tilde\ga$ of $\ga$ such that 
$\tilde\ga^{-1}(x_i)\leq \tilde\ga^{-1}(x_{i+1})$ for $0\leq i\leq k-1$. If a polygon $\si$
is inscribed in a polygon $\rho$, then $\kappa(\si)\leq\kappa(\rho)$ 
because the sum of the interior angles of triangles in CAT(0) spaces is bounded above by $\pi$. (See Lemma 2.1 in \cite{AB}.)
\begin{definition}
Let $\gamma$ be a curve in a CAT(0) space $X$. Then its {\em total curvature} $\ka(\gamma)$ is defined as
$$
\ka(\ga):=\sup\{\ka(\si)|\ \si \text{ is an inscribed polygon in }\ga\}.
$$
\end{definition}

\noindent Note that this definition generalizes the Riemannian definition of total curvature. A curve of
finite total curvature is rectifiable and has left and right directions at every interior point. Moreover, Fenchel's theorem holds,
\ i.e. the total curvature
of a closed curve $\ga$ is at least $2\pi$ and equality is attained if and only if $\gamma$ bounds a
flat convex subset or degenerates to an interval. (Cf. \cite{AB}, Section 2.)

\subsection{Lipschitz maps}

We have the following version of Rademacher's theorem which is a consequence of Proposition 1.4 and Theorem 1.6 in \cite{Lyt}.

\begin{proposition}
 Let $X$ be a CAT(0) space and $K\subset \R^n$ a measurable subset. Let $f:K\to X$ be a Lipschitz map.
 Then $f$ is almost everywhere differentiable in the following sense. For almost all $p\in K$ there exists a unique map
 $df_p:T_p\R^n\to T_{f(p)}X$  whose image is a Euclidean space, which is linear and such that
 $$
 \lim_{|v|\to 0}\frac{|f(p+v),f(p)|-|df_p(v),o_{f(p)}|}{|v|}=0
 $$
 where $o_{f(p)}$ denotes the origin in $T_{f(p)}X$.
\end{proposition}

\begin{lemma}[\cite{Kir}, Theorem 7]\label{lem:areaformula}
Let $X$ be a CAT(0) space.
 Let $K\subset \R^2$ be measurable, let  $f:K\to X$ be a Lipschitz map with $f(K)=Y$. Let 
 $N_f:Y\to[1,\infty]$ be the multiplicity function $N_f(y)=\#\{z\in K:f(z)=y\}$. Then the
 following area formula holds true.
 $$
 \int_Y N_f(y)\ d\mathcal{H}^2(y)=\int_K\mathcal{J}(df_z) dz 
 $$
 where $\mathcal{J}(df_z)$ denotes the usual Jacobian of the linear map $df_z$.
\end{lemma}

\begin{corollary}\label{cor:monarea}
 Let $X$ be a metric space.  Let $K\subset \R^2$ be measurable and  $f:K\to \R^2$ be a monotone Lipschitz map with $f(K)=Y$.
 Further let $s:Y\to X$ be a Lipschitz map. Then $\a(s\circ f)=\a(s)$.
\end{corollary}

\begin{proof}
We have $N_{s\circ f}\geq N_s$. Since $f$ is monotone, it only takes the values $1$ and $\infty$ on $Y$.
We conclude $N_{s\circ f}= N_s$ for $\mathcal{H}^2$-almost all points in $X$ and the claim follows from Lemma \ref{lem:areaformula}.
\end{proof}

We will repeatedly make use of the following proposition which is a consequence of Theorem 2.5 in \cite{ABC}.

\begin{proposition}\label{prop:ae}
Let $f:\bar D\ra\R$ be a Lipschitz map and denote by $\Pi_y:=f^{-1}(y)$ its fibers. Then for almost every value $y\in \R$ the following holds.

\begin{enumerate}
	\item $\mathcal{H}^1(\Pi_y)$ is finite.
	\item Each connected component of $\Pi_y$ which has positive length is contained in a rectifiable Jordan curve $\Gamma\subset\R^2$. Moreover, the union of those components has
full $\mathcal{H}^1$-measure in $\Pi_y$.
\end{enumerate} 
\end{proposition}

\begin{proof}
 Extend $f$ to a Lipschitz map $F:\R^2\to\R$ with compact support. Denote by $\Pi^F_y:=F^{-1}(y)$ its fibers. By Theorem 2.5 in \cite{ABC}, almost every fiber $\Pi_y^F$ of $F$
 decomposes as $\Pi_y^F=N\cup\bigcup_{i=1}^\infty \Ga_i$ where $\mathcal{H}^1(N)=0$ and each $\Ga_i$ is a rectifiable Jordan curve.
 Moreover, $\mathcal{H}^1(\Pi^F)$ is finite and equal to $\sum_{i=1}^\infty\mathcal{H}^1(\Ga_i)$. The claim follows.
\end{proof}

\begin{definition}
 Let $f:\bar D\ra\R$ be a Lipschitz map. We call $y\in\R$ an {\em quasi regular value}, if the conclusion of Proposition \ref{prop:ae}
 holds for the fiber $\Pi_y$ of $f$.
\end{definition}

\subsection{Preparation for cut and paste}

\begin{lemma}\label{lem:conecom}
 Let $X$ be a CAT(0) space, $p\in X$ a point and $r>0$ some radius. 
 Suppose that $\gamma:S^1\to X$ is a Lipschitz circle with image contained in $\bar{B}_r(p)$.
Let $v:\bar D\to X$ be the $p$-radial extension of $\gamma$. 
 Then $v$ is Lipschitz continuous with 
\[\a(v)\leq \frac{r}{2}\cdot\length(\gamma).\]
Moreover, equality holds if and only if $\gamma$ lies at constant distance $r$ from $p$ and the geodesic cone over $\gamma$ with tip $p$
is intrinsically  a flat cone.
\end{lemma}

\begin{proof}
Recall that $v$ 
maps radial geodesics in $\bar D$  with constant speed to $p$-radial geodesics in $X$.
We first show that $v$ is Lipschitz continuous by estimating distances using piecewise
radial and spherical paths. For $x\in\bar D$ set $\lambda_x=|\ga(\frac{x}{\|x\|}),p|$. If $L\geq r$ is a Lipschitz constant for $\gamma$, then we have
$$
|v(x),v(y)|\leq \lambda_x\cdot|\|x\|-\|y\||+L\cdot\angle_0(x,y)\cdot\|y\|\leq 3L\cdot\|x-y\|.
$$
At almost every point $x\in\bar D$ we can estimate the Jacobian by 
$\mathcal{J}(dv_x)\leq  \lambda_x\cdot|\dot\gamma(\frac{x}{\|x\|})|.$
Hence
\begin{align*}
\a(v)&\leq \int_0^1\int_0^{2\pi}\lambda_\theta |\dot\gamma(\theta)| s\ ds d\theta\\
&\leq r\cdot\int_0^1 s\cdot\left( \int_0^{2\pi}|\dot\gamma(\theta)|d\theta\right) ds=\frac{r}{2}\cdot\length(\gamma).  
\end{align*}

Equality implies $\lambda_\theta\equiv r$ and that the length of distance spheres in the intrinsic space grows exactly linearly.
Hence the claim follows from the rigidity statement in Bishop-Gromov (Theorem \ref{thm:BG}).
\end{proof}

\begin{lemma}\label{lem:reparaunit}
 Let $X$ be a CAT(0) space and $\ga:S^1\to X$ a Lipschitz curve of length equal to $2\pi$.
 Then the following holds.
 \begin{enumerate}
  \item For $\theta\in[0,2\pi)$ let $R_\theta:S^1\to S^1$ be the counter clockwise rotation by the angle $\theta$. 
  For each such $\theta$ there exists a Lipschitz homotopy $h_\theta:S^1\times[0,1]\to X$
  form $\ga$ to $\ga\circ R_\theta$ with $\im(h_\theta)=\im(\ga)$.
  \item Let $\bar\ga$ be an arc length parametrization of $\ga$, 
  then there is a Lipschitz homotopy $h:S^1\times[0,1]\to X$ form $\bar\ga$ to $\ga$ with $\im(h)=\im(\ga)$.
 \end{enumerate}
In particular, all of the above homotopies have zero mapping area.
\end{lemma}

\begin{proof}
 For the first part we simply set $h_\theta(e^{it},s)=\ga(e^{i(t+s\theta)})$. 
 For the second part, denote by $l(\theta)$ the length of the curve $\ga(e^{it})|_{[0,\theta]}$ and set $\tau(\theta)=e^{il(\theta)}$. 
 Note that $l:[0,2\pi]\to[0,2\pi]$ is Lipschitz continuous and 
$h(\theta,s)=\bar\ga(e^{i((1-s)\theta+s l(\theta))})$ provides a homotopy as required.
\end{proof}

\begin{lemma}\label{lem:homosmallarea}
 Let $X$ be a CAT(0) space and let $\epsilon>0$ be a number.
\begin{enumerate}
\item For $i=1,2$, let $\ga_i:[0,1]\to X$  be  Lipschitz curves with Lipschitz constant $L>0$. If 
$\sup\limits_{t\in [0,1]}|\ga_1(t),\ga_2(t)|<\epsilon$, then  
 there exists 
 a Lipschitz homotopy $h$ from $\ga_1$ to $\ga_2$ with $\a(h)\leq \frac{8}{\pi}\cdot L\cdot\epsilon$.
\item Let $\ga^+:[0,1]\to X$ be an $L$-Lipschitz curve with $L>0$ and denote by $\ga^-$ the geodesic
from $\ga^+(1)$ to $\ga^+(0)$. Suppose that $\im(\ga^+)\subset N_{\epsilon}(\im(\ga^-))$. Let $\ga$
denote the concatenation of $\ga^+$ and $\ga^-$. Then there exists 
 a Lipschitz disc $w$ filling $\ga$ and such that $\a(w)\leq \frac{8}{\pi}\cdot L\cdot\epsilon$.
\end{enumerate}
 
\end{lemma}

\begin{proof}
 Choose $n\in\N$ such that $\frac{\epsilon}{2 L}\leq\frac{1}{n}\leq\frac{\epsilon}{L}$. On each interval $\{\frac{k}{n}\}\times[0,1]$, $1\leq k\leq n$,
 we define the homotopy $h$ to be the constant speed parametrization of the geodesic from $\ga_1(\frac{k}{n})$
 to $\ga_2(\frac{k}{n})$. To extend the definition to the remaining domains in $[0,1]\times[0,1]$ we use the isoperimetric inequality \ref{thm:isoin}.
 The boundaries of these domains map to curves of length $\leq 4\epsilon$. Therefore
$h$ is a Lipschitz map of area bounded above by $n\cdot\frac{(4\epsilon)^2}{4\pi}\leq \frac{8}{\pi} L\epsilon$. This proves part i).

For part ii) we choose $n$ as above. Denote by $x_k$ the nearest point projection of $\ga^+(\frac{k}{n})$. Using the geodesics between 
$\ga^+(\frac{k}{n})$ and $x_k$ we can cut $\ga$ into $n$ Lipschitz circles of length bounded above by $4\epsilon$. We can now finish  the proof
as above, using the isoperimetric inequality.
 \end{proof}

\section{CAT(0) geometry}

\subsection{Funnel extensions of CAT(0) spaces}\label{sec:funnel}

\begin{definition}
For $\alpha>0$ denote $C_\alpha$ the Euclidean cone over a circle of length $\alpha$.
A metric space $E_\alpha$ is called a {\em flat funnel} (of angle $\alpha$), if it is isometric to 
the complement of a relatively compact convex neighborhood of the vertex of $C_\alpha$.  
\end{definition}

\begin{definition}
Let $C_\alpha$ be a Euclidean cone over a closed interval of length $\alpha$.
A metric space $s_{\alpha,r}$ is called a {\em flat sector} (of angle $\alpha\geq 0$ and radius $r>0$), if it is isometric to 
the closed $r$-ball around the vertex of $C_\alpha$. An infinite flat sector will be called an {\em ideal flat sector}. The {\em legs} $l_1$ and $l_2$ of $s_{\alpha,r}$ are the two radial geodesic segments of length $r$ lying
in the boundary of $s_{\alpha,r}$. They intersect in a single point $p$, the {\em tip} of $s_{\alpha,r}$. 
\end{definition}

\begin{lemma}\label{lem:locglu}
Let $X$ be a CAT(0) space and $q,x,y$ three points in $X$ such that $r:=|qx|=|qy|>0$ and $\alpha:=\angle_q(x,y)\leq\pi$. Denote
$l_1$ and $l_2$ the legs of a flat sector $s_{2\pi-\alpha,r}$.
If $f:l_1\cup l_2\ra X$ is an intrinsic isometry onto $qx\cup qy$, then
$X\cup_f s_{2\pi-\alpha,r}$ is a CAT(0) space.
\end{lemma}

\begin{proof}
By the theorem of Cartan-Hadamard, it is
enough to show that  $X\cup_f s_{2\pi-\alpha,r}$ is locally CAT(0). Since geodesic segments in CAT(0) spaces are convex, Reshetnyak's theorem implies the claim away from the point $q$. 
We show that the $r$-ball around $q$ is CAT(0). 

\noindent One can be obtain $B_r(q)\subset X\cup_f s_{2\pi-\alpha,r}$ from $B_r(q)\subset X$ in two steps. First we glue a flat sector $s_{\pi-\alpha,r}$ along one of its legs to the geodesic segment $qx$. In the resulting CAT(0) space the second leg of $s_{\pi-\alpha,r}$ extends the geodesic segment $qy$ to a geodesic $\sigma$ of length $2r$. 
In a second step we can now glue a flat half-disc of radius $r$ to $\si$, thereby producing 
$B_r(q)\subset X\cup_f s_{2\pi-\alpha,r}$. Hence $X\cup_f s_{2\pi-\alpha,r}$ is CAT(0).
\end{proof}

\begin{lemma}
Let $\ga$ be an embedded closed curve in $X$ with finite total curvature $\ka(\gamma)$. Then there is a flat funnel $E_{\ka(\ga)}$ and an intrinsic
isometry $f:\gamma\ra\partial E_{\ka(\ga)}$ such that $X\cup_f E_{\ka(\ga)}$ is a CAT(0) space. 
\end{lemma}

\begin{proof}
Let us assume that $\gamma$ is a k-gon. We will glue the flat funnel
$E_\alpha$ to $X$ in two steps. First, let us choose 
for every pair of adjacent vertices $x_i$ and $x_{i+1}$ of $\ga$ a flat half-strip $h_i\cong[0,|x_i,x_{i+1}|]\times [0,\infty)$. We then 
glue these flat half-strips to $X$ via isometries $[0,|x_i,x_{i+1}|]\times \{0\}\ra x_ix_{i+1}$. In the resulting space we
see two geodesic rays $r_i^-\subset h_{i-1}$ and $r_i^+\subset h_{i+1}$ emanating from every vertex $x_i$. The distance between the directions of $r_i^+$ and $r_i^-$ at $x_i$ 
is given by $\pi+\angle_{x_i}(x_{i-1},x_{i+1})$.
As a second step we insert ideal flat sectors $s_i$ of angle $\alpha_i:=\pi-\angle_{x_i}(x_{i-1},x_{i+1})$, thereby completing the angle at every vertex of $\ga$ to $2\pi$.
Altogether we obtain $X\cup E_\alpha$, and from the construction it is clear that $\alpha=\ka(\ga)$. The CAT(0) property follows from Lemma \ref{lem:locglu} together with the theorem of Cartan-Hadamard.

The case for general $\gamma$ follows by approximating $\gamma$ by inscribed polygons $\ga_n$ and then taking an ultralimit of the resulting spaces $X\cup E_{\ka(\ga_n)}$.
\end{proof}

\subsection{CAT(0) surfaces}

A CAT(0) space $Z$ which is homeomorphic to a topological surface is called a {\em CAT(0)} surface.
Since CAT(0) spaces are contractible, the topology of CAT(0) surfaces is rather simple. In particular,
any compact CAT(0) surface $Z$ is homeomorphic to the closed unit disc $\bar D$ in $\R^2$. In this case,
$Z$ is called a {\em CAT(0) disc}. If $Z$ is a CAT(0) surface and $z$ is a point in the interior of $Z$, then 
small metric balls around $z$ are homeomorphic to $\bar D$ and the link $\Sigma_z Z$ is homeomorphic to a circle.
Moreover, a metric ball around an interior point is even bilipschitz to the corresponding ball in the tangent space 
and the Lipschitz constant can be chosen arbitrary close to one as the radius of the ball tends to zero \cite{Bur}.
Using bilipschitz parametrizations we can define the area of Lipschitz maps with domain a CAT(0) surface and target an arbitrary metric space.
Moreover, the Hausdorff $2$-measure on a CAT(0) surface behaves similarly to the Lebesgue measure on a smooth Hadamard surface. 
For instance, we have the following folklore version of Bishop-Gromov's theorem, cf. Proposition 7.4 in \cite{N}.

\begin{theorem}(Bishop-Gromov)\label{thm:BG}
Let $Z$ be a CAT(0) surface. Let $p$ be a point in $Z$ and suppose that $|p,\D Z|>r$. Then we have
$\mathcal{H}^2(B_r(p))\geq\frac{r^2}{2}\cdot\mathcal{H}^1(\Sigma_p Z)$. Moreover, equality holds if and only if 
$B_r(p)$ is isometric to the radius $r$ ball around the tip in $T_p Z$.
 
\end{theorem}

Note that each CAT(0) surface is geodesically complete in the sense that any geodesic
segment is contained in geodesic segment whose boundary lies in the boundary of the surface. This is immediate from the fact
that a pointed neighborhood of an interior point cannot be contractible. Hence the interior of a CAT(0) surface is GCBA in the sense of 
Lytchak and Nagano \cite{LN}.

\begin{lemma}\label{lem:areaconv}
 Let $(Z_k)$ be a sequence of CAT(0) discs with rectifiable boundaries. Assume that $(Z_k)$  Gromov-Hausdorff converges  to a CAT(0) disc $Z$.
 If boundary lengths are uniformly bounded, $\mathcal{H}^1(\D Z_k)< C$, then the total measures converge,
 $\mathcal{H}^2(Z_k)\to\mathcal{H}^2(Z)$.
 \end{lemma}
 
 \begin{proof}
 By Theorem 12.1 in \cite{LN} it
is enough to show that no measure is concentrated near the boundary. 

Let $\epsilon>0$. Since $Z\setminus\D Z$ is an intrinsic space, we can find a Jordan polygon $\si\subset Z\setminus\D Z$ such that 
the arc length parametrization of  $\si$ is uniformly $\epsilon$-close to an arc length parametrization of $\D Z$. Then we lift $\si$
to Jordan polygons $\si_k$ in $Z_k$. Denote by $W\subset Z$ the closure of the Jordan domain of $\si$ and by $W_k$ the closure of 
the Jordan domain of $\si_k\subset Z_k$. Then $W_k\to W$ and it is enough to bound $\mathcal{H}^2(Z_k\setminus W_k)$ uniformly.
But since $\mathcal{H}^1(\D Z_k)< C$ and $\length(\si_k)\to\length(\si)$, Lemma \ref{lem:homosmallarea} below implies
$\mathcal{H}^2(Z_k\setminus W_k)\leq C'\cdot\epsilon$ with a uniform constant $C'>0$. Hence the claim follows.
 \end{proof}

If $Z$ is a CAT(0) disc, then its interior $Y:=Z\setminus\D Z$ is a length space which is still locally CAT(0).
As such it is a surface of bounded integral curvature in the sense of Alexandrov \cite{A}. For a detailed account to 
surfaces of bounded integral curvature we refer the reader to \cite{AZ}. Here we only collect a few facts needed later. 
On $Y$ there exists a nonpositive Radon measure $\mu$, called {\em curvature measure}, such that if a Jordan triangle $\triangle$
is contained in a Jordan domain $O$, then the excess of $\triangle$ is bounded below by $\mu(O)$. In particular, $\mu(O)=0$
for a Jordan domain $O\subset Y$ is equivalent to $O$ beeing flat. The atoms of $\mu$ correspond to
points in $Y$ where the tangent cone is not isometric to the flat plane. More precisely, for $y\in Y$
holds $\mu(\{y\})=2\pi-\mathcal{H}^1(\Sigma_y Y)$.

\begin{lemma}\label{lem:limitcone}
 Let $Z$ be a CAT(0) surface and let $(z_i)$ be a sequence in $Z$. Further, let $(\eps_i)$ be a nullsequence
 of positive real numbers. If $(z_i)$ converges to a point $z$ in the interior of $Z$, then the following holds.
 \begin{enumerate}
  \item $\wlim(\frac{1}{\eps_i}Z,z_i)\cong\R^2$\ if\ $\wlim\frac{|z,z_i|}{\eps_i}=+\infty$.
  \item $\wlim(\frac{1}{\eps_i}Z,z_i)\cong T_z Z$\ if\ $\wlim\frac{|z,z_i|}{\eps_i}<+\infty$.
 \end{enumerate}

\end{lemma}

\begin{proof}
 If $\wlim\frac{|z,z_i|}{\eps_i}<+\infty$, then  $\wlim(\frac{1}{\eps_i}Z,z_i)$ is isometric to $\wlim(\frac{1}{\eps_i}Z,z)$,
 although possibly not pointed-isometric. The second claim then follows from $\wlim(\frac{1}{\eps_i}Z,z)\cong T_z Z$.
 
 Set $W:=\wlim(\frac{1}{\eps_i}Z,z_i)$. Then $W$ is a complete CAT(0) surface without boundary and we need show that it is flat. 
 Notice that $\lim\limits_{r\to 0}\mu(B_r(z)\setminus\{z\})=0$. Hence for each fixed $R>0$ we have
 $\lim\limits_{i\to\infty}\mu(B_{\eps_i R}(z_i))=0$. Now if $\triangle$ is a geodesic triangle in $W$, then we can find geodesic triangles
 $\triangle_i$ in $(Z,z_i)$ such that $\wlim\frac{1}{\eps_i}\triangle_i=\triangle$ and all three angles converge (Proposition 2.4.1 in \cite{HK}). 
 Since there exists $R_0>0$ such that $\triangle_i\subset B_{\eps_i R_0}(z_i)$ for $\om$-all $i$, we see that the excess of $\triangle_i$
 goes to zero. It follows that $W$ is flat and therefore $W\cong\R^2$. 
 
\end{proof}

\begin{lemma}\label{lem:Zbilip}
 Let $Z$ be a CAT(0) disc and assume that its boundary $\D Z$ is a polygon with positive angles.
 Then $Z$ is bilipschitz to $\bar D$.
\end{lemma}

\begin{proof}
The double $Y$ of $Z$ is a closed Alexandrov surface of bounded integral curvature. By Theorem 1 in \cite{Bu} it is therefore
bilipschitz to the round sphere $S^2$. By the Lipschitz version of the Sch\"onflies theorem in \cite{Tuk} we conclude that 
$Z$ is bilipschitz to $\bar D$. 
\end{proof}

\begin{lemma}\label{lem:bilip}
 Let $Z$ be a CAT(0) surface. Let $\Ga$ be a rectifiable Jordan curve in $Z$ and denote by $\Om_\Ga$ its Jordan domain.
 Then $\bar{\Om}_\Ga$ equipped with the induced intrinsic metric is a CAT(0) disc. If $\Ga$ is a Jordan polygon with positive angles,
then this space is even bilipschitz to
 $\bar{\Om}_\Ga$.
\end{lemma}

\begin{proof}
 To proof the first claim we will use Theorem \ref{thm:plateau} and Theorem \ref{thm:int} below.
 By Theorem \ref{thm:plateau}, there is a solution $u:\bar D\to Z$ to the Plateau problem for $(\Ga,Z)$.
 By Theorem \ref{thm:int}, $u$ factors over the associated intrinisic space $Z_u$ and induces a short map
 $\bar u:Z_u\to Z$ which restricts to an arc length preserving homeomorphism $\D Z_u\to \Ga$.
 Moreover, $Z_u$ is a CAT(0) disc. Now since $Z$ is a surface, we infer from Theorem 6.1 in \cite{LWcanonical}  that $u$
 is a homeomorphism from $\bar D$ to $\bar \Om_\Ga$. This implies that $\bar u$  provides an isometry $Z_u\to\bar\Om_\Ga$.

 We turn to the second claim.
 Any metric is bounded above by its associated intrinsic metric. By compactness, it is enough to locally control the intrinsic metric 
 on $\bar{\Om}_\Ga$ by the induced metric.
 However, near an interior point both metrics coincide. But in a neighborhood of a boundary point the two metrics are still
 bilipschitz equivalent because the angles of $\Ga$ are positive. Hence the claim holds.

 \end{proof}

\begin{lemma}\label{lem:polygon}
 Let $Z$ be a CAT(0) disc with rectifiable boundary $\D Z$. Suppose that $c:S^1\to\D Z$ is a $L$-Lipschitz parametrization. 
Then for every $\epsilon>0$ there exists a $(L+\epsilon)$-Lipschitz curve $\si:S^1\to Z\setminus\D Z$ which parametrizes a Jordan polygon of positive angles and is uniformly $\epsilon$-close to $c$.
\end{lemma}

\begin{proof}
 Choose $n\in\N$ such that $\frac{L\cdot 4\pi}{n}<\frac{\epsilon}{2}$ and then choose $\delta>0$
such that $4\cdot n\cdot \delta<\frac{\epsilon}{2}$. Next, choose points $\si(\frac{k \cdot 2\pi}{n})$ in $Z\setminus\D Z$, such that 
$|\si(\frac{k \cdot 2\pi}{n}),c(\frac{k \cdot 2\pi}{n})|<\delta$. By the triangle inequality we  obtain 
$|\si(\frac{(k-1) \cdot 2\pi}{n}),\si(\frac{k \cdot 2\pi}{n})|<\frac{L\cdot 2\pi}{n}+2\cdot \delta$. Since $Z\setminus\D Z$ is a length space (Theorem 1.3 in \cite{LWcurv}),
we can define $\si|_{[\frac{(k-1) \cdot 2\pi}{n},\frac{k \cdot 2\pi}{n}]}$ to be a constant speed parametrization of a polygon in $Z\setminus\D Z$ 
whose length is bounded above by $\frac{L\cdot 2\pi}{n}+3\cdot\delta$.
Hence the Lipschitz constant of $\si$ is less than $L+\frac{3\cdot n\cdot\delta}{2\pi}$ which is less than $L+\epsilon$ by our choice of $\delta$. 

Again by the triangle inequality we obtain
$|\si(t),c(t)|<\frac{L\cdot 4\pi}{n}+4\cdot\delta<\epsilon$ and therefore $\si$ is uniformly $\epsilon$-close to $c$.

It is easy to modify $\si$ such that it becomes Jordan and has only positive angles.
For the latter we just observe that if $x,y,z$ are different points in $Z$ and 
we take any point $y'\neq x,y$ on the geodesic $xy$, then  
$\angle_{y'}(y,z)\neq 0$ by the uniqueness of geodesics. 
 \end{proof}

Recall that a map between topological spaces is called {\em monotone}, if the inverse image of every point is connected.

\begin{lemma}\label{lem:mono}
 Let $Z$ be a CAT(0) disc with rectifiable boundary $\D Z$. Let $c:S^1\to \D Z$ be a Lipschitz parametrization.
 Then $c$ extends to a Lipschitz continuous monotone map $\mu:\bar D\to Z$.
\end{lemma}

\begin{proof}
 Pick a point $p$ in the interior of $Z$. Then extend $c$ by mapping radial geodesics in $\bar D$ to constant speed $p$-radial geodesics in $Z$.
 Lipschitz continuity follows in the same way as in Lemma \ref{lem:conecom} below. 
 Observe that if $y$ and $z$ are two points in $\D Z$ such that the geodesics $py$ and $pz$ intersect in a nontrivial geodesic segment $px$,
 then the union $xy\cup xz$ separates $Z$. Hence one of the two components of $\D Z\setminus\{y,z\}$ has the property that the geodesic from any of its points to $p$
 contains $px$. This implies the claimed monotonicity.
\end{proof}

We will need to recognize when a CAT(0) surface is flat away from a finite number of cone points.
We begin with a definition.

\begin{definition}
 Let $X$ be a CAT(0) space. Fix a point $p\in X$. Then we call a point $x\in X\setminus\{p\}$ a {\em $p$-branch point}, if
 the geodesic segment $px$ extends to different geodesics $py^-$ and $py^+$ such that $py^-\cap py^+=px$.
 
 We denote by  $\mathcal{R}_r(p)$  the {\em regular star of radius $r$ around $p$}. By definition, it is the union of all $p$-radial geodesics of length $r$ which do not contain a $p$-branch point. Let
$\bar{\mathcal{R}}_r(p)$ denote its closure.
 
\end{definition}

Now let $Z$ be a CAT(0) surface and $p\in Z$ some point. Note that each $p$-branch point makes a positive contribution to the area of $\bar B_r(p)$. 
Hence there are at most countably many $p$-branch points in $\bar B_r(p)$.
$\mathcal{R}_r(p)$ is regular in the sense that each $p$-radial geodesic extends uniquely to a maximal $p$-radial geodesic of length $r$.
Since geodesics on surfaces separate locally, we see that $p$-radial geodesics in  
$\bar{\mathcal{R}}_r(p)$ extend in at most two ways. There is a natural quotient space $Q_r(p)$
associated to $\bar{\mathcal{R}}_r(p)$. By definition, two points in $\bar{\mathcal{R}}_r(p)$ are equivalent it they have 
the same distance from $p$ and they project to the same direction in $\Sigma_p Z$. So the quotient map $\pi:\bar{\mathcal{R}}_r(p)\to Q_r(p)$
identifies different extensions of $p$-radial geodesics. Clearly, $\pi$ is radially isometric with respect to $p$ and preserves Hausdorff $1$-measure on
distance spheres around $p$. It is not hard to see that $Q_r(p)$ is a CAT(0) disc. The CAT(0) property follows from Reshetnyak's gluing theorem (Theorem \ref{thm:glue}) if
there is only a finite number of $p$-branch points. The general case follows from a limit argument.

 By Bishop-Gromov (Theorem \ref{thm:BG}), $\mathcal{H}^2(\mathcal{R}_r(p))=\mathcal{H}^2(Q_r(p))\geq\frac{r^2}{2}\mathcal{H}^1(\Sigma_{\pi(p)} (Q_r(p))=\frac{r^2}{2}\mathcal{H}^1(\Sigma_p Z)$.
Moreover, equality holds if and only if $Q_r(p)$ is isometric to a Euclidean cone of cone
angle $\mathcal{H}^1(\Sigma_p Z)$; or to put it another way, if $\mathcal{R}_r(p)$ is isometric to a Euclidean cone which is cut open along some
radial geodesics.

\begin{lemma}\label{lem:conesurf}
 Let $Z$ be a CAT(0) surface and let $P=\{p_1,\ldots,p_k\}\subset Z$ be a finite subset. Let $R>0$ be such that $|p_i,\D Z|<R$ for all $p_i\in P$. Suppose that 
the regular stars $\mathcal{R}_r(p_i)$ are disjoint for all $r\leq R$ and such that 
 $\mathcal{H}^2(N_r(P))=\frac{r^2}{2}\sum_{i=1}^k \mathcal{H}^1(\Sigma_{p_i} Z)$. Then $N_R(P)$ is flat away from a finite set of cone points.
\end{lemma}

\begin{proof}
 By Bishop-Gromov (Theorem \ref{thm:BG}), our assumptions guarantee that each quotient $Q_r(p_i)$ of $\bar{\mathcal{R}}_r(p_i)$ is flat away from $p_i$. Since $\bigcup_{i=1}^k \mathcal{R}_r(p_i)$
 has full measure in $N_r(P)$, we see that $N_r(P)$ is obtained by gluing the individual regular stars along their boundaries.
 Each component of the boundary of a regluar star is a geodesic hinge. Hence the only source for nonflatness are the vertices of these hinges.
 It is therefore enough to show that each $\mathcal{R}_r(p_i)$ has only a finite number of boundary components. But each boundary component yields a
 point $z\in Z$ with $\mathcal{H}^1(\Sigma_{z} Z)\geq \frac{3\pi}{2}$. Since each compact subset of $Z$ sees only a finite number of such points, we are done.
\end{proof}

\section{Minimal discs}

\subsection{Sobolev spaces}

We will collect some basic definitions and properties from Sobolev space theory in metric spaces as developed in 
 \cite{KS},\cite{R}, \cite{HKST} and \cite{LWplateau}. For more details we refer the reader to these articles.
We denote by $\Omega$ an open bounded Lipschitz domain in the Euclidean plane and fix a complete metric space $X$.
 Following Reshetnyak (\cite{R1}), we say that a map $u:\Om\to X$ has {\em finite energy}, or lies in the Sobolev space $W^{1,2}(\Om,X)$ if
 \begin{itemize}
  \item $u$ is measurable and has essentially separable image.
  \item There exists $g\in L^2(\Om)$ such that the composition $f\circ u$ with any short map $f:X\to \mathbb{R}$ 
lies in the classical Sobolev space $W^{1,2}(\Om)$
  and the norm of the weak gradient $|\nabla(f\circ u)|$ is almost everywhere bounded above by $g$.
 \end{itemize}

 Any Sobolev map $u$ has a well defined trace $\operatorname{tr}(u)\in L^2(\partial\bar{\Om})$. (Cf. \cite{KS} and \cite{LWplateau}.)
If $u$ has a representative which extends to a continuous map $\bar u$ on $\bar{\Om}$, then $\operatorname{tr}(u)$ is represented by 
$\bar u|_{\partial\bar{\Om}}$.
If the domain $\Om$ is homeomorphic to the open unit disc $D$, then we call a map $u\in W^{1,2}(\Om,X)$ a {\em Sobolev disc}.

We say that a circle $\gamma:\partial\bar{D}\to X$ {\em bounds} a Sobolev disc $u$, if $\operatorname{tr}(u)=\gamma$ in 
$L^2(\partial\bar{D},X)$. 
 
\subsubsection{Energy and the Dirichlet problem}

Every Sobolev map $u\in W^{1,2}(\Om,X)$ has an approximate metric differential at almost every point. More precisely,
for almost every point $z\in\Om$ there exists a unique seminorm on $\R^2$, denoted $|du_z(\cdot)|$
such that 

$$
\aplim\limits_{w\to z}\frac{|u(w),u(z)|-|du_z(w-z)|}{|w-z|}=0,
$$
where $\aplim$ denotes the approximate limit, see \cite{EG}.

\begin{definition}
The {\em Reshetnyak energy} of a Sobolev map $u\in W^{1,2}(\Om,X)$ is given by
$$
E(u):=\int\limits_{\Om}\max_{v\in S^1}|du_z(v)|^2 dz.
$$
\end{definition}

 \begin{theorem}[Dirichlet problem, \cite{KS}]\label{thm:dirichlet}
 Let $\gamma$ be a circle in a CAT(0) space $X$ which can be spanned by a Sobolev disc. 
 Then there exists a unique Sobolev disc $u$ which minimizes the energy among all Sobolev discs spanning $\gamma$. 
 The energy minimizer $u$ is locally Lipschitz continuous in $D$ and extends continuously to $\bar D$.
 
 Moreover, the local Lipschitz constant of $u$ at a point $z$ depends only on the total energy of $u$ and the distance of $z$ to the boundary $\partial\bar D$.
 \end{theorem}

\subsubsection{Area and the Plateau problem}

\begin{definition}\label{def:area}
The {\em parametrized (Hausdorff) area} of a Sobolev map $u\in W^{1,2}(\Om,X)$ is given by
$$
\a(u):=\int\limits_{\Om}\mathcal{J}(du_z) dz,
$$
where  the  Jacobian $\mathcal{J}(s)$ of  a  seminorm  $s$  on $\R^2$ is the Hausdorff $2$-measure of the unit square with respect to $s$
if $s$ is a norm and $\mathcal{J}(s)=0$ otherwise.
\end{definition}

For a given Jordan curve $\Gamma$ we denote by $\Lambda(\Gamma,X)$ the family
of Sobolev discs $u\in W^{1,2}(D,X)$ whose traces have representatives which are monotone parametrizations of $\Gamma$, cf. \cite{LWplateau}.

\begin{definition}[Area-minimizer]
 Let $\Gamma$ be a Jordan curve and  $u\in\Lambda(\Gamma,X)$
 a Sobolev map. The map $u$ will be called {\em area minimizing}, if it has the least area among all
 Sobolev competitors, i.e. if
 $$
 \a(u)=\inf\{\a(u')|\ u'\in \Lambda(\Gamma,X)\}.
 $$
\end{definition}

The following theorem is a special case of Theorem 1.4 in \cite{LWplateau}.

\begin{theorem}[Plateau's problem]\label{thm:plateau}
Let $X$ be a CAT(0) space and $\Gamma\subset X$ a Jordan curve. Then there exists a Sobolev disc $u\in \Lambda(\Gamma,X)$
with
$$
E(u)=\inf\{E(u')|\ u'\in \Lambda(\Gamma,X)\}.
$$
Moreover, every such $u$ has the following properties.

\begin{enumerate}
	\item $u$ is an area minimizer.
	\item $u$ is a conformal map in the sense that there exists a {\em conformal factor} $\lambda\in L^2(D)$ with $|du_z|=\lambda(z)\cdot s_0$
	almost everywhere in $D$, where $s_0$
	denotes the Euclidean norm on $\R^2$.
	\item $u$ has a locally Lipschitz continuous representative which extends continuously to $\bar{D}$.
	
\end{enumerate} 
 
\end{theorem}

\begin{definition}
 Let $X$ be a CAT(0) space and $\Gamma\subset X$ a Jordan curve. A map $u\in \Lambda(\Gamma,X)$ as in Theorem \ref{thm:plateau} above is  called a {\em minimal disc}
 or a {\em solution of the Plateau problem for} $(\Gamma, X)$.
\end{definition}

\subsection{Intrinsic minimizers}

The following result is a consequence of Theorem 1.2 in \cite{LWcurv} and Theorem 7.1.1 in \cite{R}. 
The factorization and the fact that the intrinsic space is a CAT(0) space can also be deduced from Theorem 1.1 in \cite{PS}.

\begin{theorem}[Intrinsic structure of minimal discs]\label{thm:int}
Let $X$ be a CAT(0) space and $\Gamma\subset X$ a Jordan curve. If $u:\bar{D}\to X$ is a minimal disc
filling $\Gamma$, then the following holds. There exists a CAT(0) disc $Z_u$ such that 
$u$ factorizes as $u=\pi\circ \bar{u}$ with continuous maps $\pi:\bar{D}\to Z_u$ and $\bar{u}:Z_u\to X$. Moreover, 
\begin{enumerate}
	\item $\pi$ is monotone and restricts to an embedding on $D$. 
	\item $\pi\in \Lambda(\partial Z_u,Z_u)\subset W^{1,2}(D,Z_u)$.
	\item $\bar{u}$ is short.
	\item $\bar{u}$ restricts to an arc length preserving homeomorphism $\D Z_u\to \Gamma$.
	\item $\bar{u}$ preserves the lengths of all rectifiable curves.  
	\item For any open subset $U\subset D$ holds $\mathcal{H}^2(\pi(U))=\a(u|_{U})=\a(\pi|_U)$.
\end{enumerate}
 
\end{theorem}
 Note that Corollary \ref{cor:monarea}
 implies $\a(\bar v\circ\pi)=\a(\bar v)$ for every Lipschitz map $\bar v: Z_u\to X$. As a consequence, the induced map
$\bar u$ is area minimizing in the sense of the following definition.

\begin{definition}[Intrinsic minimal surface]
Let $Z$ be a CAT(0) surface and $X$ be a CAT(0) space.
A short map $f:Z\to X$ will be called {\em intrinsic minimal surface} or {\em intrinsic (area) minimizer}, 
if 
\begin{enumerate}
 \item $f$ is area preserving in the sense that $\a(f|_U)=\mathcal{H}^2(U)$ for every open set $U\subset Z$;
 \item  for each closed disc $Y$ embedded in $Z$
the map $f|_Y$ has the least area among all
 Lipschitz competitors, i.e. 
 $$
 \a(f|_Y)=\inf\{\a(f')|\ f':Y\to X\text{ Lipschitz with }f'|_{\D Y}=f|_{\D Y}\}.
 $$
\end{enumerate}

If $Z$ is homeomorphic to a plane or a closed disc, we call $f$ {\em intrinsic minimal plane}, respectively {\em intrinsic minimal disc}.
\end{definition}

\subsubsection{Basic properties}

\begin{lemma}\label{lem:restrict}
Let $X$ be a CAT(0) space and $f:Z\to X$ an intrinsic minimal surface. Suppose that $\Gamma\subset Z$ is a rectifiable Jordan curve with Jordan domain $\Om_\Ga$.
Then the restriction $f_{\bar\Om_\Ga}$ is an intrinsic minimal disc, where $\bar\Om_\Ga$ is equipped with the induced intrinsic metric.
\end{lemma}

\begin{proof}
 By Lemma \ref{lem:bilip} $\bar\Om_\Ga$ is a CAT(0) disc. The other properties are immediate.
\end{proof}

\begin{lemma}[Convex hull property]\label{lem:cohu}
Let $u:\bar{D}\to X$ be a minimal disc in a CAT(0) space $X$ and $p\in X$ a point. If $u(\D\bar{D})\subset \bar{B}_p(r)$, then 
$u(\bar{D})\subset\bar{B}_p(r)$.
\end{lemma}

\begin{proof}
Since the nearest point projection $\pi:X\ra\bar{B}_p(r)$ is short, the energy of $\pi\circ u$ is bounded above by the energy of $u$. 
Note that $\pi\circ u$ and $u$ have the same boundary values. Because $u$
is the unique energy minimizing filling with respect to its boundary we conclude $\pi\circ u=u$. 

\end{proof}

\begin{lemma}[Maximum principle]\label{lem:maxpr}
 Let $u:\bar{D}\to X$ be a harmonic disc in a CAT(0) space $X$. Let $\varphi:X\to\R$ be a continuous convex function. Then the function $\varphi\circ u$ attains its maximum at the boundary.
 Moreover, if the maximum is attained at an interior point, then either $\varphi$ attains a minimun and $u$ maps into the minimal level or else  $u(\bar D)\subset u(\D \bar D)$.
\end{lemma}

\begin{proof}
 By Theorem 2 b) in \cite{F}, $\varphi\circ u$ is subharmonic. Hence the maximum principle yields the first claim. 
 For the second claim, we use the strong maximum principle to conclude that $\varphi\circ u$ is constant.
 But if $u(\bar D)\not\subset u(\D \bar D)$ and $\im(u)$ is not contained in the minimum level, 
 we could decrease the energy of $u$ by locally pushing $u$ towards the minimum
 using the gradient flow of $-\varphi$. 
\end{proof}

We will make use of the following elementary observation.

\begin{lemma}\label{lem:divisionbound}
Let $\gamma:S^1\to X$ be a Lipschitz circle. Assume that $\gamma(p)=\gamma(q)$ for $p\neq q\in S^1$.
Denote $S^\pm$ the two components of $S^1\setminus\{p,q\}$ and let $\gamma^\pm:S^\pm/\partial S^\pm\to X$ be the induced loops.
Suppose that $u^\pm:\bar D\to X$ are Lipschitz discs filling $\gamma^\pm$. Then there exists a Lipschitz disc
$u:\bar D\to X$ which fills $\gamma$ and such that
\[\a(u)\leq\a(u^+)+\a(u^-).\]
\end{lemma}

\begin{proof}
Let $\pi:S^1\to S^1/\{p=q\}$ be the quotient map. Glue two discs $D^\pm$ to $S^1/\{p=q\}$
such that the resulting space $Y$ is a union of two discs which intersect in a single point.
Denote by $\iota:S^1/\{p=q\}\to Y$ the canonical embedding. Note that the mapping cylinder $\Pi$ of 
$\iota\circ\pi$ is homeomorphic to a disc. Write $\Pi=S^1\times[0,1]/(x,1)\sim(\iota\circ\pi(x),1)\cup D^+\cup D^-$.
Then we obtain a Lipschitz map $v$ defined on $\Pi$ by setting $v|_{D^\pm}=u^\pm$ and $v|_{S^1\times\{t\}}=\gamma$. 
The area of $v$ is equal to $\a(u^+)+\a(u^-)$. The desired map $u$ is then given by precomposing $v$ with the quotient map $\bar D\to \Pi$.
\end{proof}

We record a special case using the same notation as above.

\begin{corollary}\label{cor:divisionbound}
 If the image of $\gamma^-$ is a tree, then $\a(u)\leq\frac{\length(\gamma^+)^2}{4\pi}$. 
\end{corollary}
\begin{proof}
Since the filling area of a tree is equal to zero, the claim follows from Lemma \ref{lem:divisionbound} and the isoperimetric inequality \ref{thm:isoin}.
\end{proof}

The following will be used repeatedly. It is a consequence of Lemma \ref{lem:conecom}.

\begin{corollary}\label{cor:isom}
Let $f:Z\to X$ be an intrinsic minimal surface.  If $Z$ contains a closed convex subset $W$  which is isometric
to a Euclidean disc, then $f$ resticts to an isometric embedding $W\to X$.
\end{corollary}

\begin{proof}
We may assume that $W$ is isometric to the closed unit disc. 
Let $w$ be the center of $W$, i.e. $W=\bar{B}_1(w)\subset Z$.
Let $c:S^1\to Z$ be an arc length parametrization of $\partial W$.
If the distance between $f\circ c$ and $f(w)$ would be less than
one, then by Lemma \ref{lem:conecom}, $\mathcal{H}^2(W)=\a(f|_W)<\frac{1}{2}\length(f\circ c)\leq \pi$. 
A cut and paste argument based on Lemma \ref{lem:conecom} would then show that $f$ is not area minimizing.  
Hence $f\circ c$ is at constant distance one from $f(w)$ and $f$ restricts to a radial isometry
on $W$. Repeating the same argument for subdiscs of $W$ with different centers shows that $f$ is an isometric embedding. 
\end{proof}

\begin{corollary}\label{cor:smallangleinj}
 Let $X$ be a CAT(0) space and $f:C_\alpha\to X$ an intrinsic minimal plane where $C_\alpha$ is a Euclidean cone of cone angle $\alpha\geq 2\pi$. Then the following holds.
 \begin{enumerate}
  \item $f$ is a locally isometric embedding away from
 the tip $o$ of $C_\alpha$. In particular, $f$ is a radial isometry with respect to $o$.
 \item If $f(x)=f(y)$ for $x\neq y$ and $v_x,v_y\in\Sigma_o C_\alpha$ denote the directions at $o$ pointing to $x$ respectively $y$, 
 then the intrinsic distance between $v_x$ and $v_y$ is at least $2\pi$.
 \item If $\alpha<4\pi$, then $f$ is injective. If even $\alpha=2\pi$, then $f$ is an isometric embedding.
 \end{enumerate}

\end{corollary}

\begin{proof}
 Claim i) is immediate from Corollary \ref{cor:isom}. If $f$ would not be injective, then $\Sigma_{f(o)}X$ would contain a geodesic loop of length $\leq \frac{\alpha}{2}$.
 Since $\Sigma_{f(o)}X$ is CAT(1), we conclude claim ii) and the first part of iii). The supplement in the third claim follows directly from Corollary \ref{cor:isom}.
\end{proof}

\begin{remark}
 In the  case $\alpha<4\pi$ above, $f$ does not have to be an isometric embedding, as can be seen in a product of two ideal tripods.
\end{remark}

\begin{lemma}\label{lem:boundarydist}
Let $\Gamma$ be  a Jordan curve in $X$.
Let $f:Z\to X$ be an intrinsic minimal disc filling $\Gamma$. Then for $p\in Z$,
\[|p,\D Z|_Z\leq\frac{\length(\Gamma)}{2\pi}.\]
Furthermore, equality holds if and only if $Z$ is a flat disc and
$f$ is an isometric embedding.
\end{lemma}

\begin{proof}
Set $R=|p,\D Z|_Z$. Then $\length(\Gamma)\geq \mathcal{H}^1(\D \bar{B}_R(p))$ since the nearest point projection onto 
$\bar{B}_R(p)$ is short. By Bishop-Gromov (Theorem \ref{thm:BG}), we have $\mathcal{H}^1(\D \bar{B}_R(p))\geq 2\pi R$.
This proves the inequality. The case of equality follows from the rigidity statement in Bishop-Gromov together with Corollary \ref{cor:isom}.
\end{proof}

\subsubsection{Minimal vs. intrinsic minimal}

\begin{lemma}\label{lem:c+p}
Let $X$ be a CAT(0) space and $Z_i$, $i=1,2$, CAT(0) discs with rectifiable boundaries. Let $c_i:S^1\to\D Z_i$ be $L_i$-Lipschitz parametrizations and 
suppose that $f_i: Z_i\to X$ are $L$-Lipschitz maps. Then there exists a constant $C=C(L,L_1,L_2)$ such that
the following holds.  If the compositions $f_i\circ c_i$  are uniformly $\epsilon$-close to each other,  
then there exists a Lipschitz map $\tilde f_1:Z_1\to X$ with $\D \tilde f_1=\D f_1$ and
\[\a(\tilde f_1)<\a(f_2)+C\cdot\epsilon.\]
\end{lemma}

\begin{proof}
By Lemma \ref{lem:polygon}, we can choose parametrized Jordan polygons $\si_i:S^1\to Z_i\setminus\D Z_i$ which have positive angles and
 are uniformly $\epsilon$-close to $c_i$. Moreover, the Lipschitz constant of $\si_i$ is bounded above by $(L_i+\epsilon)$. 
 We denote the associated Jordan domains by $\Om_i$. 
By Lemma \ref{lem:bilip}, $\bar\Om_i$ is intrinsically
a CAT(0) disc and there is a bilipschitz map $\varphi_i:\bar D\to Z_i$. In order to obtain $\tilde f_1$ we will cut and paste $f_1|_{\bar\Om_1}$.

By Lemma \ref{lem:reparaunit}, there are Lipschitz homotopies $h_i$ of zero area between $\D (f_i\circ\varphi_i)$ and $f_i\circ\si_i$.
By our assumptions, the properties of $\si_i$ and the triangle inequality we conclude 
$\sup_{t\in S^1}|f_1\circ\si_1,f_2\circ\si_2|<(2L+1)\cdot\epsilon$. Hence Lemma \ref{lem:homosmallarea} gives a Lipschitz homotopy $h$ between
$f_1\circ\si_1$ and $f_2\circ\si_2$ with 
\[\a(h)\leq C\cdot L\cdot\max_{i=1,2}L_i\cdot\epsilon.\]
Now we define a Lipschitz map $\Phi:\bar D\to X$ with $\D \Phi=\D(f_1\circ\varphi_1)$ as follows. We partition $\bar D$ into a central disc and three concentric annuli.
Then we use $h_1$ on the outmost annulus, then $h$, then $h_2$
and on the central disc we use $f_2\circ\varphi_2$. In particular, $\a(\Phi)\leq \a(f_2)+C\cdot L\cdot\max_{i=1,2}L_i\cdot\epsilon$.

Again by Lemma \ref{lem:homosmallarea}
\[\a(f_1|_{Z_1\setminus\Om_1})\leq L^2\cdot\mathcal{H}^2(Z_1\setminus\Om_1)\leq L^2\cdot C\cdot L_1\cdot\epsilon.\]
Now we cut $f_1|_{\bar\Om_1}$ and use $\varphi_1$ to paste $\Phi$. This defines $\tilde f_1$.
Combining the estimates above gives the necessary area bound for $\tilde f_1$:
\[\a(\tilde f_1)\leq L^2\cdot C\cdot L_1\cdot\epsilon+\a(f_2)+C\cdot L\cdot\max_{i=1,2} L_i\cdot\epsilon.\]

\end{proof}

The following is a special case of Theorem 1.2 in \cite{LWcanonical}. It can also be deduced from 
Theorem 1.2 in \cite{PS2} and Corollary \ref{cor:monarea}. 

\begin{lemma}\label{lem:eqarea}
 Let $Z$ be a CAT(0) disc and $\Ga\subset Z$ a rectifiable Jordan curve with Jordan domain $\Om_\Ga$.  Suppose that $u:\bar D\to Z$ is a minimal disc filling $\Ga$.
 Then $\im(u)=\bar\Om_\Ga$, $\a(u)=\mathcal{H}^2(\Om_\Ga)$ and  for any Lipschitz map $f:Z\to Y$ to a metric space $Y$ holds
 $\a(f\circ u)=\a(f|_{\Om_\Ga})$.
\end{lemma}

 By Theorem \ref{thm:int}, every minimal disc yields an intrinsic minimizer. The following proposition provides a converse.
 
 \begin{proposition}\label{prop:confsol}
  Let $X$ be a CAT(0) space and  $f:Z\to X$ an intrinsic minimal surface. Suppose that $\Gamma\subset Z$
  is a rectifiable Jordan curve. Let
  $u:D\to Z$ be a minimal disc filling $\Gamma$. Then $f\circ u$ is  conformal and harmonic. If in addition $f$ restricts to an embedding on $\Gamma$,
	then $f\circ u$ is a solution to the Plateau problem for $(f(\Gamma),X)$.
  
\end{proposition}

 \begin{proof}
  Let us settle the claim on conformality first.
  For almost every $x\in D$ we have $|du_x|=\lambda\cdot s_0$ where $s_0$ denotes the Euclidean norm; $df_{u(x)}$ is a linear isometric embedding;
  $u$ is differentiable at $x$ with a linear differential and the chain  rule holds \cite{Lyt}. Hence $f\circ u$ is conformal.
  
 We will show that $f\circ u$ is harmonic. Choose a small $\rho>0$ and set $u_\rho:=u|_{(1-\rho)\cdot\bar D}$. Let $v_\rho\in W^{1,2}((1-\rho)\cdot D,X)$
 be a solution to the Dirichlet problem with $\operatorname{tr}(v_\rho)=\operatorname{tr}(f\circ u_\rho)$. By Theorem \ref{thm:dirichlet}, $v_\rho$ extends continuously to $(1-\rho)\cdot\bar D$.
 The extension will still be called $v_\rho$. In particular, $\D v_\rho=\D (f\circ u_\rho)$. Then we have 
 \[\a(v_\rho)\leq E(v_\rho)\leq E(f\circ u_\rho)=\a(f\circ u_\rho)=\a(f|_{\im(u_\rho)}).\]
 Now $u_\rho$ is a Lipschitz embedding. Therefore $Z_\rho:=\overline{\im(u_\rho)}$ is intrinsically a CAT(0) disc (Lemma \ref{lem:bilip}) 
 and $f_\rho:=f|_{Z_\rho}$ is an intrinsic minimizer (Lemma \ref{lem:restrict}).
 
 By Proposition 3.1 in \cite{LWY}, for every $\epsilon>0$ there exists a Lipschitz disc $\tilde v_\rho:(1-\rho)\cdot\bar D\to X$ with $\D \tilde v_\rho=\D v_\rho= f\circ \D u_\rho$ and 
 $\a(\tilde v_\rho)\leq \a( v_\rho)+\epsilon$. Since $\D u_\rho$ is a Lipschitz parametrization of $\D Z_\rho$, Lemma \ref{lem:c+p} implies
 $\a(f_\rho)\leq\a(v_\rho)$. Hence $E(v_\rho)= E(f\circ u_\rho)$ and by uniqueness $v_\rho=f\circ u_\rho$. As $\rho>0$ was arbitrary,
 we conlude that $f\circ u$ is harmonic.
 
 If $f|_\Gamma$ is an embedding, then we use Lemma \ref{lem:c+p} and argue as above to show that $f\circ u$ is area minimizing which completes the proof.
 \end{proof}

 \begin{corollary}\label{cor:strong}
  Let $X$ be a CAT(0) space and let $f:Z\to X$ be intrinsic minimal surface. Let $\Gamma\subset Z$ be a rectifiable Jordan curve and denote by $\Om_\Ga$
  its Jordan domain. Let $\varphi:X\to\R$ be a continuous convex function. If $\varphi\circ f|_{\bar\Om_\Ga}$  attains its maximum in $\Om_\Ga$, then 
  $\varphi$ attains a minimum and $\im(f)$ is contained in the minimum level of $\varphi$.
  
 \end{corollary}
 
 \begin{proof}
  By Lemma \ref{lem:restrict}, $f|_{\bar\Om_\Ga}$ is an intrinsic minimal disc. Therefore we may assume that $Z$ is a disc and $\Ga=\D Z$.
  Let $u$ solve the Plateau problem for $(\Ga,Z)$. By Theorem 1.1 in \cite{LWcanonical}, $u$ is a homeomorphism. Hence if  $\varphi\circ f$
  attains a maximum in $\Om_\Ga$, then $\varphi\circ f\circ u$ attains a maximum in $D$.
  By Proposition \ref{prop:confsol},
  $f\circ u$ is harmonic and hence we can apply the  maximum principle (Lemma \ref{lem:maxpr}) to $\varphi\circ f\circ u$ and conclude that
  either our claim holds or else
  $\varphi\circ f\circ u(\bar D)\subset \varphi\circ f\circ u(\D\bar D)$. But this leads to $\a(f)=\a(f\circ u)=0$. Contradiction. 
  \end{proof}

  \begin{corollary}\label{cor:separated}
  Let $X$ be a CAT(0) space and let $f:Z\to X$ be a intrinsic minimal surface. Let $p\in X$ be a point and set $f_p:=|f(\cdot),p|$.
  Let $r>0$ be a quasi regular value of $f_p$ with $r<|p,f(\D Z)|$. Suppose that $\Pi_r=N\cup\bigcup_{i=1}^\infty \Ga_i$ is the corresponding
  decomposition of the fiber, cf. Proposition \ref{prop:ae}. Then the associated Jordan domains $\Om_i$ are all disjoint.
 \end{corollary}
 
 \begin{proof}
 Assume that $\Om_1\subset \Om_2$. Then $f_p|_{\bar \Om_2}$ attains its maximal value $r$ in $\Om_2$.
  The distance function $|\cdot,p|$ is continuous and convex. It has a unique minimum at $p$. Hence Corollary \ref{cor:strong} implies that
  $f$ is constant equal to $p$ on $\Om_2$. Contradiction.
  \end{proof}

\subsection{Limits of minimal discs}

 \begin{lemma}\label{lem:liftarea}
 Let $(X_k,x_k)$ be a sequence of pointed CAT(0) spaces and denote by $(X_\om,x_\om)$ their ultralimit.
 Let $(Z_k)$ be a sequence of CAT(0) discs which Gromov-Hausdorff converge to a CAT(0) disc
 $Z$.  Assume that each  $Z_k$ is bilipschitz to $\bar D$ and that the boundary lengths $\mathcal{H}^1(\D Z_k)$ are uniformly bounded. 
 Suppose that $f:Z\to X_\om$ is a Lipschitz map and set $\ga:=\partial f$. Further assume that for some $L>0$ there are $L$-Lipschitz circles $\ga_k:\D Z_k\to X_k$ with $\wlim \ga_k=\ga$. 
 Then, for every $\eps>0$ there exist Lipschitz maps
 $f_k:Z_k\to X_k$ with $\D f_k=\ga_k$ and such that for $\om$-all $k$ holds
 $$
 \a(f_k)\leq \a(f)+\eps.
 $$
\end{lemma}

\begin{proof}
Let $\rho:S^1\to\D Z$ be a constant speed parametrization and set $c:=\ga\circ\rho$. By assumption, we can find constant speed
parametrizations $\rho_k:S^1\to\D Z_k$ such that $c_k:=\ga_k\circ\rho_k$ $\om$-converges to $c$. By Lemma \ref{lem:mono}, $\rho$
extends to a monotone Lipschitz map $\mu:\bar D\to Z$. We put $v:=f\circ \mu$. By Corollary \ref{cor:monarea}, we have $\a(v)=\a(f)$.
For given $\eps>0$, Theorem 5.1 in \cite{W} provides a sequence $(v_k)$ of Lipschitz maps $v_k:\bar D\to X_k$ filling $c_k$ and such that 
$\a(v_k)\leq\a(v)+\eps$ holds for $\om$-all $k$. The statement follows since each $Z_k$ is bilipschitz to $\bar D$.
\end{proof}

\begin{remark}
 The condition on the $Z_k$ is necessary because if $\D Z_k$ has a peak, then there might not be a single Lipschitz map $Z_k\to X_k$
 filling $c_k$.
\end{remark}

\begin{proposition}\label{prop:limitminimizer}
 Let $(X_k,x_k)$ be a sequence of pointed CAT(0) spaces and denote by $(X_\om,x_\om)$ their ultralimit. 
 Further, let $(Z_k,z_k)$ be a sequence of CAT(0) discs, each bilipschitz to $\bar D$. Suppose that $(Z_k)$ Gromov-Hausdorff converges to a CAT(0) disc $Z$
 and such that the boundary lengths $\mathcal{H}^1(\D Z_k)$ are uniformly bounded. 
 For each $k\in\N$ let $f_k:Z_k\to X_k$ be an intrinsic minimal disc with $f(z_k)=x_k$.
 Then $f_\om:=\wlim f_k :Z_\om\to X_\om$ is an intrinsic minimal disc
 with $\a(f_\om)=\wlim \a(f_k)$.
\end{proposition}

\begin{remark}
 If we remove the condition on $Z_\om$ beeing a disc, then $f_\om$ is still area minimizing and its domain is a CAT(0)
 disc retract, cf. \cite{PS}.
\end{remark}

 \begin{proof}
  Since all the $f_k$ are short, we obtain a well defined short limit map 
  $f_\om:Z_\om\to X_\om$. Since each $f_k$ is area minimizing, we conclude from Lemma \ref{lem:liftarea} $\a(\varphi)\geq \wlim\a(f_k)$
  for any Lipschitz map $\varphi$ with $\D\varphi=\D f_\om$. 
  By our assumption, $Z_\om$ is isometric to $Z$. From Lemma \ref{lem:areaconv} we know
  $\mathcal{H}^2(Z)=\lim\limits_{k\to\infty}\mathcal{H}^2(Z_k)$. Hence
  $$
  \a(f_\om)\leq\mathcal{H}^2(Z_\om)=\lim\limits_{k\to\infty}\mathcal{H}^2(Z_k)=\wlim\a(f_k).
  $$
  Therefore, equality holds and $f_\om$ is an intrinsic minimal disc.
  
 \end{proof}

\subsection{Monotonicity}

A key property of minimal surfaces in smooth spaces is the monotonicity of area ratios.
The aim of this section is to prove  monotonicity in a more general setting.

\begin{lemma}\label{lem:intconecom}
Let $X$ be a CAT(0) space and let $f:Z\to X$ be an intrinsic minimal disc. Suppose that there is a point $p$
in $X$ and a raduis $r>0$ such that $f(\D Z)\subset\D \bar B_r(p)$. Then $a(f)\leq\frac{r}{2}\cdot \mathcal{H}^1(\D Z)$.
 
\end{lemma}

\begin{proof}
 Let $\epsilon>0$. By Lemma \ref{lem:polygon}, we find a parametrized Jordan polygon $\si:S^1\to Z$ with positive angles which is uniformly close to an arc length parametrization of $\D Z$
 and such that  $\length(\si)\leq(1+\epsilon)\cdot \mathcal{H}^1(\D Z)$ holds.
 Denote by $\Om_\si$ the associated Jordan domain. By Lemma \ref{lem:homosmallarea}, we have 
 \[\a(f|_{Z\setminus\Om_\si})=\mathcal{H}^2(Z\setminus\Om_\si)\leq C\cdot\mathcal{H}^1(\D Z)\cdot\epsilon\]
 with a uniform constant $C>0$. By  Lemma \ref{lem:bilip}, $\bar\Om_\si$ is intrinsically a CAT(0) disc which is bilipschitz to $\bar D$. Hence Lemma \ref{lem:conecom} implies
 \[\a(f|_{\Om_\si})\leq \frac{r}{2}\cdot\length(\si)\leq\frac{(1+\epsilon)\cdot r}{2}\cdot\mathcal{H}^1(\D Z).\]
 The claim follows since $\epsilon>0$ was arbitrary. 
 
\end{proof}

\begin{proposition}[Intrinsic monotonicity]\label{prop:mon}
Let $X$ be a CAT(0) space and $Z$ a CAT(0) surface. Suppose that $f:Z\to X$ is an intrinsic minimizer. Then for any point $p\in X$
the area density
$$
\Theta(f,p,r):=\frac{\mathcal{H}^2(f^{-1}(B_r(p)))}{\pi r^2}
$$
is a nondecreasing function of $r$ as long as
$r<|p,f(\partial Z)|$. 
 
\end{proposition}

\begin{proof}
We put $f_p(x):=|f(x),p|$ and define $\Omega_r:=f_p^{-1}([0,r))$ and $\Pi_r:=f_p^{-1}(r)$. 
Moreover, we set $A(r):=\mathcal{H}^2(\Omega_r)$ and $L(r):=\mathcal{H}^1(\Pi_r)$. 
Then, since $f_p$ is short, the coarea formula yields $A(r)\geq \int_{\Om_r}|\nabla f_p|=\int_0^r L(t) dt$.
Therefore 
\begin{equation}
A'(r)\geq L(r)
\label{eq:low}
\end{equation}
for almost all $r<|p,f(\D Z)|$.

The desired monotonicity follows, if we can show that 
\begin{equation}
A(r)\leq\frac{r}{2}L(r)
\label{eq:up}
\end{equation}
holds almost everywhere.
By Proposition \ref{prop:ae}, almost all $r$ are quasi regular. By Corollary \ref{cor:separated}, all Jordan domains resulting from a decomposition of a quasi regular fiber are disjoint. 
Hence we may assume that $\Pi_r$ is equal to a single rectifiable Jordan curve. By Lemma \ref{lem:restrict}, $f|_{\bar\Om_r}$ is an intrinsic minimal disc and therefore the required area estimate follows 
from Lemma \ref{lem:intconecom}.
\end{proof}

\begin{corollary}[Monotonicity]\label{cor:mon}
Let $X$ be a CAT(0) space. Suppose that $u:\bar{D}\to X$ is a minimal disc and $p$
is a point in $u(\bar{D})\setminus u(\partial\bar{D})$. Then the area density
$$
\Theta(u,p,r):=\frac{\a(u(D)\cap B_r(p))}{\pi r^2}
$$
is a nondecreasing function of $r$ as long as
$r<|p,u(\partial\bar D)|$.  
 
\end{corollary}

\begin{proof}
Factorize $u$ as $\bar u\circ \pi$ as in Theorem \ref{thm:int}. Then 
$\a(u(D)\cap B_r(p))=\mathcal{H}^2(\bar u^{-1}(B_r(p)))$, by Theorem \ref{thm:int} vi).
Since $\bar u$ is an intrinsic minimizer,  Proposition \ref{prop:mon} applies.
\end{proof}

\subsection{Densities and blow-ups}

The monotonicity of area densities justifies the following definition.

\begin{definition}(Density)
For an intrinsic area minimizer $f$ and a point $p\in f(Z)\setminus f({\D Z})$ we define the {\em density at $p$}
by 
$$
\Theta(f,p):=\lim_{r\to 0}\Theta(f,p,r).
$$ 
\end{definition}

If $Z$ is compact, then the density is finite.
The function $p\mapsto \Theta(f,p)$ is upper semi-continuous by  monotonicity (Proposition \ref{prop:mon}).

\begin{lemma}\label{lem:estinv}
Let $X$ be a CAT(0) space. 
If $f:Z\to X$ is an intrinsic area minimizer, and $p\in f(Z)\setminus f({\D Z})$, then 
$$
\Theta(f,p)\geq \# f^{-1}(p).
$$
\end{lemma}

\begin{proof}
Let $\{x_1,\ldots,x_k\}$ be a finite subset of the inverse image of the point $p$ under $f$.
For $r<\frac{1}{2}\min\{|x_i,x_j|\ |\ 1\leq i<j\leq k\}$, the balls $B_r(x_i)$ are disjoint and since 
$f$ is short, we have $\a (f^{-1}(B_r(p)))\geq\sum_{i=1}^k\a (B_r(x_i))$. The claim follows from Bishop-Gromov (Theorem \ref{thm:BG}). 
\end{proof}

\begin{corollary}\label{cor:finitefiber}
 Let $X$ be a CAT(0) space and $f:Z\to X$ an intrinsic minimizer of finite area, $\a(f)<\infty$.
 Then the fiber of each point $p\in f(Z)\setminus f({\D Z})$ is finite.
\end{corollary}

\begin{proof}
 For $r<|p,f(\D Z)|$ holds  $\Theta(f,p,r)\leq \frac{\a(f)}{\pi r^2}$. Hence the claim follows from monotonicity and
 Lemma \ref{lem:estinv}.
\end{proof}

\begin{definition}
Let $\om$ be a nonprincipal ultrafilter on $\N$.
 Let $X$ be a CAT(0) space and let $f:Z\to X$ be an intrinsic minimizer.  Fix a point $z_0\in Z$. For any $r>0$ we define the rescaled maps $f_r:(r\cdot Z,z_0)\to (r\cdot X,x_0)$
 where $x_0=f(z_0)$. A {\em tangent map at $z_0$} is an ultralimit of rescalings $f_{\frac{1}{\eps_i}}$ for some nullsequence $(\epsilon _i)$:
 $$
 df_{z_0}:T_{z_0}Z\to (X_\om,x_\om).
 $$
 Here $X_\om$ denotes the pointed ultralimit $\wlim(\frac{1}{\eps_i}\cdot X,x_0)$.
 
\end{definition}

\begin{lemma}\label{lem:blowup}
 Let $X$ be a CAT(0) space and $f:Z\to X$ an intrinsic minimizer. Let $df_{z_0}:T_{z_0}Z\to X_\om$ be a tangent map at a point $z_0$
 in the interior of $Z$.
 Then $df_{z_0}$ is an intrinsic minimal plane which is a locally isometric embedding away from the tip  $o_{z_0}$ of $T_{z_0}Z$. 
 In particular, it is of constant area density $\Theta(df_{z_0},x_\om,r)\equiv\frac{\mathcal{H}^1(\Sigma_{z_0}Z)}{2\pi}$. 
 Moreover, if $x_0\in f(Z)\setminus f(\D Z)$ and $f^{-1}(x_0)=\{z_0,\ldots,z_k\}$ is a finite fiber, then  $\sum_{i=0}^k\Theta(df_{z_i},x_\om)= \Theta(f,x_0)$.
\end{lemma}

\begin{proof}
 Since $z_0$ lies in the interior of $Z$, the tangent cone $T_{z_0}Z$ is isometric to a Euclidean cone $C_\alpha$ with $\alpha=\mathcal{H}^1(\Sigma_{z_0}Z)\geq 2\pi$.
 For each $r>0$ and $i$ large enough, $f|_{B_{\eps_i r}(z_0)}$ is an intrinsic minimal disc in $X$ and $\bar B_{\eps_i r}$ is bilipschitz to $\bar D$. 
 By Proposition \ref{prop:limitminimizer}, we conclude that $df_{z_0}|_{B_r(o_{z_0})}$
 is an intrinsic minimal disc. Hence $df_{z_0}$ is an intrinsic minimal plane. By Corollary \ref{cor:isom}, $df_{z_0}$ is a locally isometric embedding away from $o_{z_0}$. In particular, it is
 radially isometric and therefore has constant area density $\Theta(df_{z_0},x_\om,r)\equiv\frac{\mathcal{H}^1(\Sigma_{z_0}Z)}{2\pi}$.
 
 Now assume that the fiber of $x_0$ is finite, $f^{-1}(x_0)=\{z_0,\ldots,z_k\}$. (By Corollary \ref{cor:finitefiber}, this is automatic if 
$Z$ is compact.) 
For simplicity we assume that $k=0$ so that $z_0$ is the only inverse image of $x_0$. The proof for $k>0$ is identical.
 We know that $\Theta(df_{z_0},x_\om)\leq \Theta(f,x_0)$ since $f$ is short and therefore $f(B_{\eps_i r}(z_0))\subset B_{\eps_i}(x_0)$.
 To see the converse inequality, let $r>0$ be such that $\eps_i r$ is quasi-regular for $\om$-all $i$. Choose Jordan domains $\Om_{\eps_i r}\subset Z$ with $z_0\in\Om_{\eps_i r}$ such that
 $\mathcal{H}^2(f^{-1}(B_{\eps_i r}(x_0)))=\mathcal{H}^2(\Om_{\eps_i r})$. We claim that
 $\wlim\frac{1}{\eps_i}\Om_{\eps_i r}=\bar B_{r}(o_{z_0})$.
 
 The inclusion ``$\supset$'' is clear, since $f$ is short. On the other hand $df_{z_0}(\wlim\frac{1}{\eps_i}\Om_{\eps_i r})\subset \bar B_{r}(o_{f(x_0)})\subset X_\om$ and $df_{z_0}$ is a radial isometry.
 Hence $\bar B_{r}(o_{z_0})=df_{z_0}^{-1}(\bar B_{r}(o_{f(x_0)}))$. Now define $r_i>0$ to be the smallest radius such that $\Om_{\eps_i r}\subset\bar B_{r_i}(z_0)$. In particular,
 $\mathcal{H}^2(\Om_{\eps_i r})\leq\mathcal{H}^2(\bar B_{r_i}(z_0))$. From our claim we obtain $\wlim \frac{1}{\eps_i}\cdot\bar B_{r_i}(z_0)=\bar B_{r}(o_{z_0})$ and in particular $\wlim\frac{r_i}{\eps_i}=r$.
 Since a subsequence of the $\frac{1}{\eps_i}\cdot\bar B_{r_i}(z_0)$ converges Gromov-Hausdorff to $\bar B_{r}(o_{z_0})$ we get from Lemma \ref{lem:areaconv} 
 $\lim_{i\to\infty}\mathcal{H}^2(\frac{1}{\eps_i}\cdot\bar B_{r_i}(z_0))=\mathcal{H}^2(\bar B_{r}(o_{z_0}))=\mathcal{H}^1(\Sigma_{z_0}Z)\cdot\frac{r^2}{2}$.
 This shows $\Theta(f,x_0)\leq \Theta(df_{z_0},x_\om)$ and completes the proof.
 
\end{proof}

 From Lemma \ref{lem:blowup} above and Proposition 1.1 in \cite{Lytopen}, we can conclude that an intrinsic minimal surface $f:Z\to X$ is
 a locally bilipschitz embedding on an open dense set of $Z$. However, our situation is more special and we actually get:
 
 \begin{theorem}
  Let $X$ be a CAT(0) space and $f:Z\to X$ an intrinsic minimizer. 
  If $z_0$ is a point in the interior of $Z$ with $\mathcal{H}^1(\Sigma_{z_0}Z)<4\pi$, then $f$ restricts to a bilipschitz embedding on a 
  neighborhood of $z_0$. In particular, if $Z$ is a CAT(0) disc, then $f$ is locally a bilipschitz embedding in the interior of $Z$ away from 
  finitely many points.
 \end{theorem}

 \begin{proof}
  Let $z_0$ be a point in the interior of $Z$ with $\mathcal{H}^1(\Sigma_{z_0}Z)<4\pi$. Assume that the claim is false. Then we can find sequences
  $(x_k)$ and $(y_k)$ in $Z$ with $x_k\neq y_k$, and such that $x_k,y_k\in B_{\frac{1}{k}}(z_0)$ and $|f(x_k),f(y_k)|\leq\frac{1}{k}|x_k,y_k|$. We set
  $\eps_k:=|x_k,y_k|$. We consider the rescaled maps $f_{\frac{1}{\eps_k}}:(\frac{1}{\eps_k}\cdot Z,x_k)\to (\frac{1}{\eps_k}\cdot X,f(x_k))$ and build the blow up
  $f_\om:\wlim(\frac{1}{\eps_k}\cdot Z,x_k)\to \wlim(\frac{1}{\eps_k}\cdot X,f(x_k))$. By Lemma \ref{lem:limitcone}, we know that $\wlim(\frac{1}{\eps_k} Z,x_k)$
  is isometric to a Euclidean cone $C_\alpha$ with cone angle $\alpha\leq\mathcal{H}^1(\Sigma_{z_0}Z) $. (Note that the tip of $C_\alpha$  might be different from $x_\om$.) 
  As in the proof of Lemma \ref{lem:blowup}, we conclude from Proposition \ref{prop:limitminimizer}
  that $df_{x_\om}$ is an intrinsic minimal plane. Corollary \ref{cor:smallangleinj} implies that $df_{x_\om}$ is injective. But by construction we have  $f_\om(x_\om)=f_\om(y_\om)$
  and $|x_\om,y_\om|=1$. This contradiction completes the proof.
  
 \end{proof}

Combining Theorem \ref{thm:structure} with Theorem \ref{thm:int}, we obtain our main structure result:

\begin{theorem}\label{thm:mainstructure}
Let $X$ be a CAT(0) space and $\Ga\subset X$ a rectifiable Jordan curve. Let $u:D\to X$ be a minimal disc filling $\Ga$.
Then there exists a finite set $B\subset D$ such that $u$ is a local embedding on $D\setminus B$.
\end{theorem}

\begin{corollary}\label{cor:inbilipext}
 Let $X$ be a CAT(0) space and $\Ga\subset X$ a rectifiable Jordan curve. Suppose that $u:D\to X$ is a minimal disc filling $\Ga$.
Denote by $Y$ the image of $u$. Then there is a finite set $P$ in $Y$ such that on $Y\setminus P$ the intrinsic and extrinsic metrics
are locally bilipschitz equivalent.
\end{corollary}

\subsection{Rigidity}

We will make use of the following auxiliary lemma which gives a lower bound on the size of links in domains of intrinsic minimizers.

\begin{lemma}\label{lem:coneangle}
Let $f:Z\to X$ be an intrinsic area minimizer.
Let  $x_1\neq x_2$ and $y$ be points in the interior of $Z$ with $f(x_i)=p, i=1, 2$ and $f(y)=q$.
Assume that $f$ maps the geodesics $x_i y$ isometrically onto the geodesic $p q$. If $v_i$ denotes the direction at $y$
pointing to $x_i$, then $|v_1,v_2|\geq 2\pi$.
\end{lemma}

\begin{proof}
Recall that $T_y Z$ is a Euclidean cone of cone angle $\alpha\geq 2\pi$. By Lemma \ref{lem:blowup}, each tangent map $df_y$ is an intrinisic
minimal plane. Hence the claim follows from Corollary \ref{cor:smallangleinj}.
  
\end{proof}

\begin{theorem}[Rigidity]\label{thm:rigid}
Let $X$ be a CAT(0) space. Let $f:Z\to X$ be an intrinsic minimal disc  and let $p$ be a point in 
$f(Z)\setminus f(\D Z)$. Assume that there exists $\Theta>0$ and a radius $R>0$ with $R<|p,f(\D Z)|$ such that
the area ratio with respect to $p$ is constant,
$$
A(r)\equiv \frac{\Theta}{2}r^2\text{  for all  }r\leq R.
$$
Then $\Om_R:=f^{-1}(B_R(p))$ is flat away from a finite set of cone points $z_1,\ldots z_k$.
Moreover,
$f$ is a locally isometric embedding on $\Om_R\setminus\{z_1,\ldots,z_k\}$ and $f(\Om_R)$
is a geodesic cone.

\end{theorem}

\begin{proof}
The supplement follows from Corollary \ref{cor:isom}, so it is enough to show that $\Om_R$ is flat
away form finitely many points. By Corollary \ref{cor:finitefiber}, $f$ has finite fibers. 
Let $P:=\{x_1,\ldots,x_n\}$ be the finitely many inverse images of $p$.
We claim that it is enough to show that the regular stars $\mathcal{R}_R(x_i)$, $i=1\ldots,n$
are disjoint. Indeed, if this holds, then we get 
$$
 \Theta=\frac{\mathcal{H}^2(\Om_R)}{\pi R^2}\geq\sum_{i=1}^n\frac{\mathcal{H}^2(\mathcal{R}_R(x_i))}{\pi R^2}
 \geq \sum_{i=1}^n\frac{\mathcal{H}^1(\Si_{x_i}Z)}{2\pi}=\Theta
$$
 The last equality follows from Lemma \ref{lem:blowup}. Since $\bigcup_{i=1}^n \mathcal{R}_R(x_i)\subset N_R(P)\subset\Om_R$
 we obtain $\mathcal{H}^2(N_R(P))=\frac{R^2}{2}\sum_{i=1}^n \mathcal{H}^1(\Si_{x_i}Z)$ and Lemma \ref{lem:conesurf} applies.
 
 To see that the regular stars are disjoint we let $m_{ij}$ denote the midpoint of $x_i$ and $x_j$. Further,
 we will denote by $v_i$ the direction at $m_{ij}$ pointing at $x_i$. If we can show $|v_i,v_j|\geq 2\pi$, then
 clearly the regluar stars have to be disjoint. Suppose this is not the case for $m_{12}$. Moreover, we may assume that it holds
 for all $(i,j)$ with $|x_i,x_j|<|x_1,x_2|$. Set $r:=\frac{|x_1,x_2|}{2}$. Then, the $\mathcal{R}_r(x_i)$ are disjoint
 and $f$ restricts to a radial isometry on each of them. It follows that $f$ maps the geodesics $x_1 m_{12}$ and $x_2 m_{12}$
 isometrically to the geodesic $p f(m_{12})$. Hence Lemma \ref{lem:coneangle} shows $|v_1,v_2|\geq 2\pi$. Contradiction.
\end{proof}

\begin{remark}
 Notably, the proof shows that the link at any midpoint $m_{ij}$ as above has length at least $4\pi$.
\end{remark}

\subsection{Extending minimal discs to planes}

Recall that a map between metric spaces is called {\em (metrically) proper} if inverse images of bounded sets are bounded.

Let $X$ be a CAT(0) space and $\hat f:\hat Z\to \hat X$ a proper intrinsic minimal plane. Then we know that the monotonicity of area densities holds for all times. More precisely,
for all points $p\in \hat f(Z)$ the function $r\mapsto\Theta(\hat f,p,r)$ is nondecreasing for all $r>0$.

\begin{definition}(Area-growth)
Let $\hat f:\hat Z\to \hat X$ be a proper intrinsic minimal plane. Then we define the {\em density at infinity} or {\em area growth} of $\hat f$
by
$$
\Theta^\infty(\hat f):=\lim_{r\to \infty}\Theta(\hat f,p,r).
$$ 
We say that $\hat f$ is of {\em quadratic area growth}, if $\Theta^\infty(\hat f)\in(0,\infty)$.
\end{definition}

Combinig the monotonicity of area densities with Lemma \ref{lem:estinv} we obtain the following.

\begin{lemma}(Key estimate)\label{lem:key}
Let $\hat f:\hat Z\to \hat X$ be a proper intrinsic minimal plane in an arbitrary CAT(0) space $\hat X$.
 Then for every point $p$ in the image of $\hat f$ we have
\begin{equation*}
\#\hat{f}^{-1}(p)\leq\Theta^\infty(\hat f).
\end{equation*}
In particular, if $\hat f$ is of quadratic area growth, then it has finite fibers.
\end{lemma}

\begin{corollary}\label{lem:eucareagrowth}
 Let $\hat X$ be a CAT(0) space and $\hat f:\hat Z\to\hat X$ a proper intrinsic minimal plane. Suppose that the area growth of  $\hat f$
 satisfies $\Theta^\infty(\hat f)<2$. Then $\hat f$ is an embedding. If the area growth is even Euclidean, $\Theta^\infty(\hat f)=1$, then
 $Z$ is isometric to the flat Euclidean plane and $\hat f$ is an isometric embedding.
\end{corollary}

\begin{proof}
The first claim is immediate from Lemma \ref{lem:key}. We turn to the second claim.
Since $f$ is short, the area growth of $Z$ is bounded above by  $\Theta^\infty(\hat f)=1$. Hence by Bishop-Gromov (Theorem \ref{thm:BG}),
$Z$ is isometric to the flat Euclidean plane. Then $\hat f$ is an isometric embedding by Corollary \ref{cor:isom}.
\end{proof}

Our goal now is to extend a minimal disc to a minimal plane such that the area growth of the minimal plane is controlled by the total curvature of the boundary of the minimal disc. If we can do this, then 
we can argue as above to control the mapping behavior of the minimal disc.

So let $\Ga$ be a Jordan curve of finite total curvature $\kappa$ in a CAT(0) space $X$. 
Set $E_\kappa:=E_{\frac{\kappa}{2\pi}}$ and denote by $\hat{X}:=X\cup_\Ga E_\kappa$ the CAT(0) space
obtained from the funnel construction, see Section \ref{sec:funnel}. Let $f:Z\to X$ be an intrinsic minimal disc spanning $\Gamma$ and such that 
$f$ restricts to a homeomorphism $\D Z\to \Gamma$. Then we can glue 
the space $\hat Z:=Z\cup_{\D Z}E_\kappa$ via $f$. Using the identity map on $E_\kappa$, we obtain a natural extension
\[\hat f:\hat Z\to\hat X.\]
Clearly, this is a proper, area preserving short map.

\begin{lemma}\label{lem:extamin}
The map $\hat f$ is area minimizing. More precisely, if $Y$ is an embedded disc in $\hat Z$, then
$\hat f|_{Y}$ minimizes the area among all Lipschitz maps $h:Y\to\hat X$ with $\D h=\D \hat f|_{Y}$.
\end{lemma}
\begin{proof}
Denote by $d_X$ the distance function to $X\subset\hat X$. Let $r_0>0$ and set $A_0:=d_X^{-1}((0,r_0])$, $Y_{0}:=\hat f^{-1}(A_0)\cup Z$.
It is enough to show that $\hat f|_{Y_0}$ is area minimizing for all large $r_0$.
Note that 
\[\a(\hat f|_{Y_0})=\mathcal{H}^2(A_0)+\a(f).\]
Since there is a monotone map  $\bar D\to Y_{0}$ (cf. Lemma \ref{lem:mono}), it is enough to show that a solution $u$ to the Plateau problem of $(\Gamma_{0},\hat X)$
has at least the area of $\hat f|_{Y_{0}}$. For topological reasons, any continuous disc filling $\Ga_{0}$ has to contain 
$A_0$ in its image. Therefore, it is enough to show that the part of $u$ which maps to $X$ has at least the area of $f$.
Let $\epsilon>0$ be a small quasi regular value of $d_X\circ u$. By Proposition \ref{prop:ae}, the corresponding fiber $\Pi_\eps$ decomposes as $\Pi_\eps=N_\eps\cup\bigcup_{k=1}^\infty \Ga_k$ where
$N_\eps$ has $\mathcal{H}^1$-measure zero and each $\Ga_k\subset D$ is a rectifiable Jordan curve. Because $u$ is minimizing, each $u_*[\Ga_i]$ has to be nontrivial in the first homology group with integer coefficients,  
$H_1(E_\kappa)$.
Since $u$ is locally Lipschitz continuous, the decomposition can only contain a finite number of Jordan curves, say $\Ga_1,\ldots,\Ga_n$. Then the sum $\sum_{i=1}^n u_*[\Ga_i]$ is equal to $u_*[\D \bar D]$ in $H_1(E_\kappa)$.
Choose Lipschitz maps $v_i:\bar D\to \bar D$ which extend arc length parametrizations of $\Ga_i$ and such that $\a(u\circ v_i)\leq \a(u|_{\Om_i})+\frac{\eps}{n}$. Since each
$u\circ v_i|_{\D \bar D}$ maps onto $\Ga_{\eps}:=d_X^{-1}(\eps)$, we can construct a Lipschitz map $v:\bar D\to \hat X$ which fills $\Ga_\eps$
and such that $\a(v)\leq\sum_{i=1}^n \a(u|_{\Om_i})+\epsilon$ and $v|_{\D \bar D}$ represents a generator of $H_1(E_\kappa)$. By Lemma \ref{lem:reparaunit}, we can adjust the boundary parametrization
to obtain a new Lipschitz disc $\tilde v:\bar D\to\hat X$ with $\a(\tilde v)=\a(v)$ and such that $\tilde v|_{\D \bar D}$ is uniformly $\eps$-close to an arc length parametrization of $\Ga$.
By minimality of $f$ and Lemma \ref{lem:homosmallarea}, we obtain $\a(v)\geq\a(f)-C\cdot\eps$ with a uniform constant $C>0$. We conclude
\[\a(u)\geq\mathcal{H}^2(A_0)+\a(f)-(C+1)\epsilon.\]
This holds for every small quasi regular value $\epsilon$ and therefore finishes the proof.
\end{proof}

\begin{lemma}\label{lem:extmin}
The extended space $\hat Z$ is a CAT(0) plane.
\end{lemma}
\begin{proof}
 Note that $\hat Z$ is a geodesic space which is homeomorphic to the plane. Moreover, $\hat Z$ contains $Z$ as a closed convex subset
since the nearest point projection $E\to\D E$ is short.

By Corollary 1.5 in \cite{LWcanonical} it is enough to show that any Jordan domain $D_c$ in $\hat Z$ bounded by a Jordan curve $c$
satisfies
\[\mathcal{H}^2(D_c)\leq \frac{\length(c)^2}{4\pi}.\]

Let $c$ be a rectifiable Jordan curve in $\hat Z$ and denote by $D_c$ the associated Jordan domain.
By Lemma \ref{lem:extmin} we know that $\hat f|_{D_c}$ is an area minimizing filling of $\hat f\circ c$ in $\hat X$. But since $\hat X$
is CAT(0), the Euclidean isoperimetric inequality holds and therefore $\a(\hat f|_{D_c})\leq\frac{\length(\hat f\circ c)^2}{4\pi}$. Since $\hat f$ is short and
preserves area, the claim follows.
\end{proof}

Hence, from Lemma \ref{lem:extamin} and Lemma \ref{lem:extmin} we obtain:

\begin{proposition}\label{prop:extension}
Let $\Ga$ be a Jordan curve of finite total curvature $\kappa$ in a CAT(0) space $X$. Denote by $\hat{X}:=X\cup_\Ga E_\kappa$ the CAT(0) space
obtained from the funnel construction. (See Section \ref{sec:funnel}.) Let  $f:Z\to X$ be an intrinsic minimal disc filling $\Gamma$. Then 
$\hat Z:=Z\cup_f E_\kappa$ is a CAT(0) plane and $f$ extends canonically to a proper intrinsic minimal plane 
$\hat f:\hat Z\to \hat X$. Moreover, $\hat f$ has area growth
$\Theta^\infty(\hat f)$ equal to $\frac{\kappa}{2\pi}$.
\end{proposition}

This proposition allows us to relate the total curvature of the boundary curve to
the multiplicity of points via our key estimate Lemma \ref{lem:key}.

\section{The F\'{a}ry-Milnor Theorem}

In this last part we will apply the above results on minimal discs 
filling curves of finite total curvature in order to obtain the general version of the F\'{a}ry-Milnor Theorem.

\begin{theorem}(F\'{a}ry-Milnor)
Let $\Ga$ be a Jordan curve in a CAT(0) space $X$. If $\kappa(\Ga)\leq 4\pi$,
then either $\Ga$ bounds an embedded disc, or $\kappa(\Ga)= 4\pi$ and $\Ga$
bounds an intrinsically flat geodesic cone. More precisely, there is a map from a convex subset of a Euclidean cone of cone angle
equal to $4\pi$ which is a local isometric embedding away from the cone point and which fills $\Ga$.
\end{theorem}

\begin{proof}
Let $u:D\to X$ be a solution to the Plateau problem for $(\Ga,X)$ provided by Theorem \ref{thm:plateau}. Denote by $f:=\bar u:Z\to X$
the induced intrinisic minimal disc. Then, by Theorem \ref{thm:int}, $\D f:\D Z\to \Ga$ is an arc length preserving homeomorphism.
Next, let $\hat X:=X\cup_\Ga E_\kappa$ 
be the CAT(0) space obtained from the funnel construction, see Section \ref{sec:funnel}. By Proposition
\ref{prop:extension} we can extend $f$ to an intrinsic minimal plane $\hat f:\hat Z\to\hat X$
with area growth $\Theta^\infty(\hat f)$ equal to $\frac{\kappa}{2\pi}$.
Then, for any point $p$ in the image of $\hat f$ our key estimate Lemma \ref{lem:key} reads 
\begin{equation}\label{eq:key}
\#\hat{f}^{-1}(p)\leq\frac{\kappa}{2\pi}.
\end{equation}
Now if $\ka<4\pi$ holds, then $\hat f$ is injective. Hence  $\Gamma$ bounds the embedded disc $f(Z)$.

So we may assume that $\ka=4\pi$ and our
intrinsic minimal disc $f$ filling $\Ga$ is not embedded. Hence we find a point $p\in\im(\hat f)$ where equality in (\ref{eq:key}) holds, i.e. which has exactly two inverse images $x^+$ and $x^-$. 
Moreover, by monotonicity (Proposition \ref{prop:mon}), we must have

$$
A(r)=\frac{\Theta^\infty(\hat f)}{2}r^2
$$
for all $r>0$. Therefore the conclusions of Theorem \ref{thm:rigid} apply
to $\hat f$. In particluar, there exists a finite set $B\subset\hat Z$ such that $\hat Z\setminus B$ is flat and the restriction of $\hat f$ is a locally isometric embedding. 
It remains to show that $B$ consists only of a single point where the cone angle is equal to $4\pi$.
To see this, let $m$ be the midpoint of $x^+$ and $x^-$. Set $r:=\frac{|x^+,x^-|}{2}$. Then $\mathcal{H}^2(B_r(x^\pm))=\pi r^2$ and
$B_r(x^\pm)\subset \hat Z$ is a flat disc. Hence, by Corollary \ref{cor:isom}, the restriction $\hat f|_{B_r(x^\pm)}$ is an isometric embedding. Consequently,
the geodesic $x^+x^-$ is folded onto the geodesic $\hat f(m)p$. 
From Lemma \ref{lem:coneangle}  we conclude that $\mathcal{H}^1(\Si_m \hat Z)\geq 4\pi$. But the area growth of $\hat Z$
is equal to $2$. Hence from Bishop-Gromov (Theorem \ref{thm:BG}) we conclude that $\hat Z$ is isometric to a Euclidean cone of cone angle $4\pi$ and the proof is complete.

\end{proof}

\newpage

\begin{example}\label{ex}

Let $X$ be the CAT(0) space which results from gluing two flat planes along a flat sector of angle $\alpha\leq \pi$.
Then $X$ contains a Jordan curve of total curvature
equal to $4\pi$ but which does not bound an embedded disc.
The picture shows an example in the case $\alpha=\pi$. 
Note that $\Gamma$, as shown in the picture, surrounds some points in $X$ twice and therefore cannot bound an embedded disc. 

\begin{figure}

\begin{center}
\includegraphics[width=0.5\textwidth]{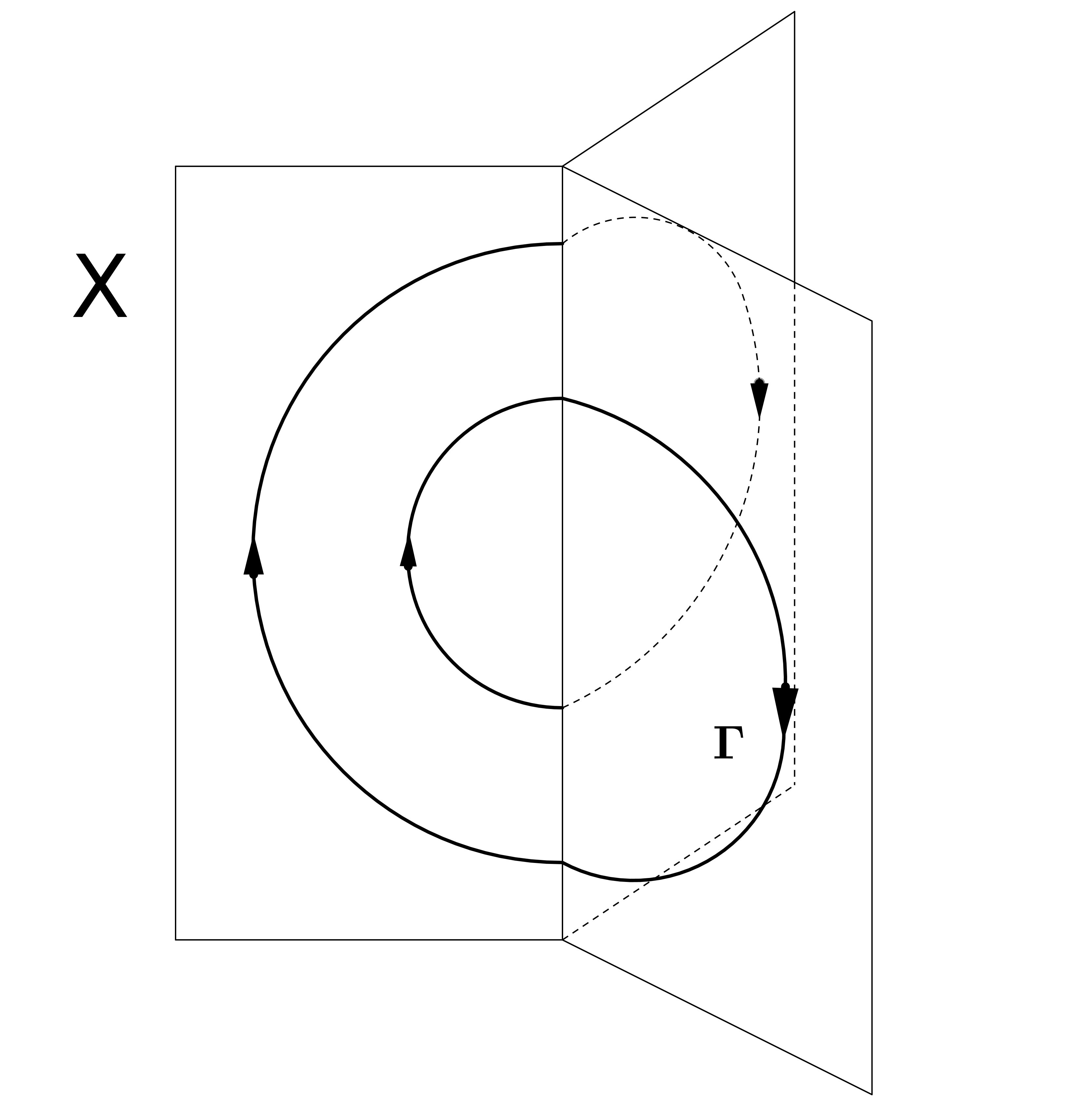}

\end{center}

\end{figure}

\end{example}

\Addresses


\begin{thebibliography}{alpha}                                  


																														
\bibitem[ABC13] {ABC} G. Alberti, S. Bianchini, G. Crippa \textit{Structure of level sets and Sard-type properties
of Lipschitz maps}, Ann. Sc. Norm. Super. Pisa Cl. Sci. (5)
vol. 12 (2013), no. 4, 863-902.
	
	
	
	
																														%
\bibitem[AB98] {AB} S. Alexander,  R. L. Bishop, \textit{The F\'{a}ry-Milnor theorem in Hadamard manifolds}, Proc. Amer. Math. Soc. 126 (1998) 3427-3436.

\bibitem[AKP17]{AKP} S. Alexander, V. Kapovitch and A. Petrunin.
\textit{Alexandrov geometry.}
http://anton-petrunin.github.io/book/all.pdf, 2017.

\bibitem[AZ67]{AZ} A.D. Aleksandrov, V.A. Zalgaller.
 \textit{Intrinsic Geometry of Surfaces.}
 AMS Transl. Math. Monographs, Vol. 15, Providence, RI, 1967.

\bibitem[A57]{A} A. D. Alexandrow.\textit{\"Uber eine Verallgemeinerung der Riemannschen Geometrie.}
Schr. Forschungsinst. Math., 1:33-84, 1957.


\bibitem[AK00]{AK} L. Ambrosio, B. Kirchheim. \textit{Currents in metric spaces}. Acta Math. 185
(2000), no. 1, 1-80.


\bibitem[B95] {Bnpc} W. Ballmann, \textit{Lectures on spaces of nonpositive curvature}, DMV-Seminar
notes, vol. 25, Birkh\"auser 1995.

\bibitem[B04]{Bmet} W. Ballmann. 
\textit{On the geometry of metric spaces.}
Preprint, lecture notes,
http://people.mpim-bonn.mpg.de/hwbllmnn/archiv/sin40827.pdf, 2004.

\bibitem[BH99]{BH} M. Bridson and A. Haefliger. \textit{Metric spaces of non-positive curvature}, volume 319 of
Grundlehren der Mathematischen Wissenschaften. Springer-Verlag, Berlin, 1999.

\bibitem[BBI01]{BBI} D. Burago, Y. Burago and S. Ivanov.
\textit{A course in metric geometry.}
Volume 33 of
Graduate Studies in Mathematics. American Mathematical
Society, Providence, RI, 2001.

\bibitem[Bu65]{Bur} Y.  Burago. \textit{On proportional approximation of a metric.}
Trudy Mat. Inst. Steklov., 76:120-123, 1965.

\bibitem[Bu05]{Bu} Y. Burago. \textit{Bilipschitz equivalent Aleksandrov srfaces. II}, Algebra i Analiz, 16:6 (2004), 28-52; St. Petersburg Math. J., 16:6 (2005), 943-960.


 \bibitem[CG03]{CG} J. Choe, R. Gulliver.  \textit{Embedded minimal surfaces and total curvature of curves in a manifold.} 
Math. Res. Lett. 10 (2003), no. 2-3, 343-362.

\bibitem[DM12]{DM} G. Daskalopoulos, C. Mese.
\textit{Monotonicity properties of harmonic maps into NPC spaces.} J. Fixed Point Theory Appl. 11 (2012), no. 2, 225-243.

\bibitem[EWW02] {EWW} T. Ekholm, B. White and D. Wienholtz. \textit{Embeddedness of minimal surfaces with total boundary curvature at most $4\pi$}, Ann. of Math. (2) 155, no. 1, 209-234 (2002).


















\bibitem[EG92]{EG} L. C. Evans, R. F. Gariepy.
\textit{Measure theory and fine properties of functions.}
Studies in Advanced Mathematics. CRC Press, Boca Raton, FL,
1992.

\bibitem[Fa49]{Fa} I. F\'{a}ry. \textit{Sur la courbure totale d'une courbe gauche faisant un noeud}, Bulletin de la Soc.
Math. de France 77 (1949), 128-138. MR 11:393h

\bibitem[F05]{F} B. Fuglede. \textit{The Dirichlet problem for harmonic maps from Riemannian
polyhedra to spaces of upper bounded curvature.}
Trans. Amer. Math. Soc., 357(2):757-792, 2005.


\bibitem[GS92]{GS} M. Gromov, R. Schoen. \textit{Harmonic maps into singular spaces and p-adic superrigidity for lattices in groups of rank one.} Publications Math\'ematiques de l'IH\'ES 76.1 (1992): 165-246.

\bibitem[GW17]{GW} C. Guo, S. Wenger. 
\textit{Area minimizing discs in locally non-compact metric spaces.}
Preprint \texttt{arXiv:1701.06736}, 2017.

\bibitem[HKST15]{HKST} J.  Heinonen, P. Koskela, N. Shanmugalingam  and Jeremy Tyson.
\textit{Sobolev spaces on metric measure spaces.}
Cambridge University Press,
2015.



\bibitem[HK05]{HK} C. Hruska, B. Kleiner. \textit{Hadamard spaces with isolated flats}, Geom. Topol. 9 (2005)
1501-1538, with an appendix by M. Hindawi, G. C. Hruska and B. Kleiner.

\bibitem[J94]{J} J. Jost.
\textit{Equilibrium maps between metric spaces.}
Calc. Var. Partial Differential Equations, 2 (2):173-204, 1994.

\bibitem[Ki94]{Kir}B. Kirchheim. 
\textit{Rectifiable metric spaces: local structure and regularity of the Hausdorff
measure.}
Proc. Amer. Math. Soc., 121(1):113-123, 1994.


\bibitem[KL97]{KL} B. Kleiner, B. Leeb, \textit{Rigidity of quasi-isometries for symmetric spaces and Euclidean buildings}, Publications Math\'ematiques de l'IH\'ES 86 (1997) 115-197 MR1608566






\bibitem[KS93]{KS} N. J. Korevaar, R. M. Schoen. 
\textit{Sobolev spaces and harmonic maps for metric space targets.} 
Comm. Anal. Geom., 1(3-4):561-659, 1993.


\bibitem[Ku96]{Ku} E. Kuwert.  \textit{Harmonic maps between flat surfaces with conical singularities.} Math. Z. 221 (1996), no. 3, 421-436.




\bibitem[Lyt04]{Lyt} A. Lytchak. \textit{Differentiation in metric spaces.} Algebra i Analiz 16
(2004), no. 6, 128-161. MR 2117451 (2005j:53039)

\bibitem[Lyt05]{Lytopen} A. Lytchak. \textit{Open map theorem for metric spaces.}
Algebra i Analiz, 17(3):139-159, 2005.



\bibitem[LN18]{LN} A. Lytchak, K. Nagano. \textit{Geodesically complete spaces with an upper curvature bound.}
Preprint\texttt{arXiv:1804.05189}, 2018.

\bibitem[LW17]{LWcanonical} A.  Lytchak, S. Wenger. \textit{Canonical  parametrizations of metric discs.}
Preprint \texttt{arXiv:1701.06346}, 2017.


\bibitem[LW16]{LWint} A.  Lytchak, S. Wenger. 
\textit{Intrinsic structure of minimal discs in metric spaces.}
Geom. Topol., to appear. Preprint \texttt{arXiv:1602.06755}, 2016.

\bibitem[LW16a]{LWcurv} A.  Lytchak, S. Wenger. 
\textit{Isoperimetric characterization of upper curvature bounds.}
Preprint \texttt{arXiv:1611.05261}, 2016.

\bibitem[LW17a]{LWplateau} A.  Lytchak, S. Wenger.
\textit{Area minimizing discs in metric spaces.}
Arch. Ration. Mech. Anal., 223(3):1123-1182, 2017.

\bibitem[LWY16]{LWY} A.  Lytchak, S. Wenger, and R. Young. 
\textit{Dehn functions and H\"older extensions in asymptotic cones.}
preprint arXiv:1608.00082, 2016.


\bibitem[Me01]{M}  C. Mese.  \textit{The curvature of minimal surfaces in singular spaces.} Comm. Anal. Geom. 9 (2001), no. 1, 3-34.

\bibitem[Mil50]{Mil} J. W. Milnor. \textit{On the Total Curvature of Knots.}
Ann. of Math. 52 (2) (1950), 248-257. MR 12:273c

\bibitem[N02]{N}    K. Nagano. \textit{A volume convergence theorem for Alexandrov spaces with curvature bounded above.}
Math. Z., 241(1):127-163, 2002.

\bibitem[P99]{P} A. Petrunin.
\textit{Metric minimizing surfaces.}
Electron. Res. Announc. Amer. Math. Soc., 5:47-54 (electronic), 1999.

\bibitem[P07]{Psemi} A. Petrunin.
\textit{Semiconcave functions in  Alexandrov's geometry.} 
Surveys in Differential Geometry, no. 92, 11, Int. Press, Somerville, MA, (2007), 137-201.

\bibitem[PS17]{PS} A. Petrunin, S. Stadler.
\textit{Metric minimizing surfaces revisited.}
Preprint \texttt{arXiv:1707.09635}, 2017.

\bibitem[PS17a]{PS2} A. Petrunin, S. Stadler.
\textit{Monotonicity of saddle maps.}
Geometriae Dedicata. 10.1007/s10711-018-0337-2. 

\bibitem[R93]{R} Y. G. Reshetnyak. 
\textit{Two-dimensional manifolds of bounded curvature.}
In Geometry, IV, volume 70 of Encyclopaedia Math. Sci., pages 3-163,
245-250. Springer, Berlin, 1993.

\bibitem[R04]{R1} Y. G. Reshetnyak.
\textit{Sobolev classes of functions with values in a metric space. II}, Sibirsk. Mat. Zh.
45 (2004), no. 4, 855-870. MR 2091651 (2005e:46055)

\bibitem[T80]{Tuk} P. Tukia. \textit{The planar Sch\"onflies theorem for Lipschitz maps.}
Ann. Acad. Sci. Fenn. Ser. A I Math., 5(1):49-72, 1980.

\bibitem[W17]{W} S. Wenger. \textit{Spaces with asymptotically Euclidean Dehn function.}
Preprint \texttt{arXiv:1707.01398}, 2017.
\end{thebibliography}
\end{document}